\documentclass{amsart}
\usepackage{mathpreamble}

\newcommand{\s}{\widetilde{s}}
\newcommand{\g}{\widetilde{g}}
\newcommand{\M}{\widetilde{M}}

\begin{document}

\title{Negatively curved K\"ahler metrics on total spaces of a class of vector bundles}

\author{Hanyu Wu}
\address{School of Mathematical Sciences, Xiamen University, Xiamen, Fujian, 361005, China.}
\email{{19020230157192@stu.xmu.edu.cn}}
\author{Bo Yang}
\thanks{The second named author is partially supported by National Natural Science Foundation of China
with the grant numbers: 11801475, 12141101, and 12271451, and Natural Science Foundation of Fujian Province of China with the grant No.2019J05012.}
\address{School of Mathematical Sciences, Xiamen University, Xiamen, Fujian, 361005, China.}
\email{{boyang@xmu.edu.cn}}
\date{11/16/2025}

\begin{abstract}
In this paper we show an abundance of complete K\"ahler metrics with negative holomorphic bisectional curvature on total spaces of certain vector bundles. Assume that such total spaces are endowed with a wider class of nonpositively curved K\"ahler metrics. We prove dimension estimates on holomorphic functions on these manifolds, as well as Liouville theorems for holomorphic mappings between them.
\end{abstract}

\subjclass[2020]{32Q05, 32Q15, 53C55}

\maketitle

\markleft{Negative curvature on total spaces of a class of vector bundles}
\markright{Negative curvature on total spaces of a class of vector bundles}

\setcounter{tocdepth}{1}
\tableofcontents

\section{Introduction}

\subsection{Statement of main results} Motivated by the classical uniformization theorem for Riemann surfaces, Greene-Wu \cite{GW77}, Siu \cite{Siu}, and Yau \cite{Yau1991} proposed to study the function theory of complete K\"ahler manifolds with a suitable negative (or nonpositive) curvature condition. Let $BI(M, g)<0$ ($BI(M, g) \leq 0$ or $BI \leq 0$ for short) denote the negative (nonpositive) holomorphic bisectional curvature of a K\"ahler manifold $(M, g)$. Many problems on complete noncompact K\"ahler manifolds with $BI<0$ remain to be investigated. We refer to \cite{GW77}, \cite{Siu}, and \cite{Yau1991} for further discussions.

We begin with two standard examples and a result due to Seshadri-Zheng \cite{SZ}.

(1) $\mathbb{C}^n$ (as a product complex structure) admits a complete K\"ahler metric with negative sectional curvature. We refer to Appendix \ref{negUN} for more discussions, including Seshadri's example on \cite[p.488]{Sesha} and a family of expanding K\"ahler-Ricci solitons constructed by Cao \cite{C97}.

(2) The unit ball $\mathbb{B}^n \subset \mathbb{C}^n$ has a natural holomorphic submersion to $\mathbb{B}^{n-1}$ by forgetting one coordinate. In this case, the base, the fiber, and the total space all admit complete K\"ahler metrics with $BI<0$.

(3) In \cite{SZ}, the authors study the non-existence of complete K\"ahler metrics with bisectional curvature pinched by two negative constants ($-C_2 \leq BI \leq -C_1<0$) on complex product manifolds. As an application, they proved that the total space of a locally trivial holomorphic fiber bundle with positive-dimensional fibers does not admit any complete K\"ahler metrics with $-C_2 \leq BI \leq -C_1<0$.

In this work, we study complete K\"ahler metrics with $BI<0$ on total spaces of certain vector bundles over compact K\"ahler manifolds. In some sense, the manifolds in study belong to a class which lies in between the  examples (1) and (2) above. Specifically, we construct complete K\"ahler metrics with $BI<0$ on total spaces of certain vector bundles over some negatively curved compact K\"ahler manifolds. The construction is achieved by the Calabi Anstaz which originated in \cite{Calabi1}, and it demonstrates an abundance of complete K\"ahler metrics with $BI<0$. Since Calabi's results (\cite{Calabi1} and \cite{Calabi2}), the Calabi method has been powerful in the construction of extremal K\"ahler metrics and other canonical K\"ahler metrics. In our case, there are no defining equations of the desired K\"ahler metrics. The key role is played by a suitable monotone quantity which captures the $BI<0$ condition. In fact, this monotone quantity originates in the study of $U(n)$-invariant K\"ahler metrics on $\mathbb{C}^n$ with $BI<0$ (see Appendix \ref{negUN}).

\begin{theorem}\label{main1}
Let $(M, g)$ be a compact K\"ahler manifold with $BI(M, g)<0$. Assume there is a holomorphic line bundle $L \rightarrow M$ with a Hermitian metric $h$ whose curvature form is $-\lambda g$ for a constant $\lambda>0$. Let $X$ be the total space of a holomorphic vector bundle $E$ in the form of $E=L^{\oplus r}$ with $r \geq 1$. We consider a smooth K\"ahler metric $\widetilde{g}$ on $X$ in the form
\begin{align}
\widetilde{\omega}=\pi^{\ast}\omega_g+\sqrt{-1}\partial\overline{\partial} u.  \label{themcalabi}
\end{align}
Here the Hermitian structure on $E$ is induced from $(L, h)$. For simplicity, let $(E, h)$ be the corresponding Hermitian metric. For any $(p, v) \in X$, we define $u(p, v)=u(h(v))$ as a smooth function on $X$ which only depends on $h$.

Then there is a natural correspondence between the following two objects
\begin{enumerate}[label=(\roman*)]
 \item  A complete K\"ahler metric $\widetilde{g}$ with $BI<0$ ($BI \leq 0$) in the form of (\ref{themcalabi});
 \item  A real-valued function $\chi(t) \in C^{\infty}[0, +\infty)$ satisfying $\chi(0)=0$ and $\chi^{\prime}(t)<0$ ($\chi^{\prime}(t) \leq 0$) for any $t \in [0, +\infty)$.
\end{enumerate}
\end{theorem}

As an application of Theorem \ref{main1}, we have following examples of complete K\"ahler manifolds with $BI<0$.

(1) Let $(M, g)$ be a compact ball quotient with its K\"ahler-Einstein metric $g$ so that $Ric(g)=-g$, and $L=K_{M}^{-k}$ for any $k \geq 1$. Moreover, there are also examples of compact ball quotients satisfy $K_M=-dL$ for some integer $d \geq 2$ with $L \in \operatorname{Pic}(M)$. Then we may choose a Hermitian metric $h$ on $L$ so that its local component is $(\operatorname{det}g)^{-\frac{1}{d}}$. For example, there is a fake projective plane $M$ which satisfies $K_M=-3L$ for some $L \in \operatorname{Pic}(M)$. Theorem \ref{main1} applies in these cases.

(2) Let $M$ be a complex submanifold of a smooth compact ball quotient $(N, g_{B})$ with the canonical K\"ahler-Einstein metric $Ric(g_B)=-g_B$. Assume that $M$ is endowed with the induced metric $g$. Let $L$ be the restriction of $K_N^{-1}$ along $M$. We consider $h$ to be the restriction of the Hermitian metric $\widetilde{h}$ on $K_N^{-1}$ as in Example (1). Since $\Theta (L, \widetilde{h})=-\sqrt{-1}\partial \bar{\partial} \ln \widetilde{h}=-\omega_{g_B}$, we get $\Theta (L, h)=-\sqrt{-1}\partial \bar{\partial} \ln h=-\omega_{g}$. This verifies the assumption of Theorem \ref{main1}.

(3) Let $M$ be a complete intersection in $\mathbb{CP}^n$ defined by $n-d$ equations of homogeneous degree $k$, According to a remarkable result due to Mohsen \cite{Mohsen}, if $n \geq 4d-1$, for sufficiently large $k$, there exists such a manifold $M$ so that the induced Fubini-Study metric satisfies $BI<0$. If $d \geq 3$, by the Lefchetz hyperplane theorem, $H^2(M, \mathbb{Z})=\mathbb{Z}$. Therefore $h^{1,1}=1$ and $M$ admits a Hodge metric $g$ with $BI<0$. In other words, $M$ admits a line bundle $(L, h)$ with its curvature $\Theta(L, h)=-\omega_g$. By Theorem \ref{main1}, we get a simply-connected complete noncompact K\"ahler manifold $(X, \widetilde{g})$ with $BI<0$. Note that $X$ is not Stein. This is related to a question proposed by Wong, see {\cite[p.100]{Wong80}}.

(4) Let $(M, g)$ be a compact K\"ahler manifold with $BI(M, g) \leq 0$ and a line bundle $L \rightarrow M$ as in Theorem \ref{main1}. We may derive analogous results as in Theorem \ref{main1} for any complete K\"ahler metric in the form of (\ref{themcalabi}) with $BI<0$ away from the zero section of $E$ (or $BI \leq 0$ on $X$). In particular, we may choose $(M, g)$ as a flat complex torus and use theta functions to construct a line bundle $L \rightarrow M$ with $c_1(L)=-\lambda g$ for some constant $\lambda>0$. We refer to \cite[p.310]{GH} or \cite[p.49]{Debarre} for more information.

(5) Even in the case that $M$ is a compact ball quotient with $L=K_{M}^{-1}$ so that the corresponding total space $X$ is $K(\pi, 1)$, the resulting metric on $X$ as in Theorem \ref{main1} may not have nonpositive Riemannian sectional curvature. We refer to Example \ref{euclideanvol} for more information.

\begin{remark}
Theorem \ref{main1} is reminiscent of a result of Bishop-O'Neill on Riemannian sectional curvature for a warped product $X=M \times_{f} F$. They consider the metric $ds_X^2=ds_M^2+f^2 ds_F^2$ for some $0<f \in C^{\infty}(M)$. We refer to {\cite[Theorem 7.5 on p.26]{BO}} for a detailed statement. In particular, one of their results states that if $(M, ds_M^2)$ has negative sectional curvature, $f$ is strictly convex, and $F$ has the real dimension $1$, then $(X, ds_X^2)$ has negative sectional curvature.
\end{remark}

\begin{remark}
In Theorem \ref{main1}, we assume that the curvature form of the line bundle $L$ is a negative multiple of the K\"ahler metric $g$ on $M$. This condition effectively leads to a nice formula of the curvature tensor of $\widetilde{g}$ (defined in (\ref{themcalabi})), see Lemma \ref{lcurvature} for a detailed statement. After a preliminary version of the paper was written, we learned from Xueyuan Wan that for any Griffiths negative vector bundle $(E, h)$ over a compact K\"ahler manifold $(M, g)$, the full curvature tensor of $\widetilde{g}$ when $u$ is chosen as $h(v)$ was solved in \cite[Proposition 3.1]{KWZ2025}. In particular, they show that such a curvature tensor vanishes along the fiber direction. In the line bundle case, such a metric satisfies the equality case of (\ref{dimcmain1}) in Proposition \ref{upperdim}.
\end{remark}

\begin{remark}
Consider a holomorphic fibration $f: X \rightarrow M$ where $M$ is compact. If $X$ is also compact, we assume in addition that $f$ is not isotrival (compare with Seshadri-Zheng's result on {\cite[p.146]{SZ}}). Suppose that both the base $M$ and each fiber of $f$ admit complete K\"ahler metrics with $BI<0$. Does $X$ admit complete K\"ahler metrics with $BI<0$? Such a question was proposed by To-Yeung \cite[p.512]{TY2011} in their study of negatively curved K\"ahler metrics on Kodaira fibration surfaces. We refer to Cheung \cite{Cheung}, Tsai \cite{Tsai}, To-Yeung \cite{TY2011}, and Wan \cite{Wan2024} for related results when $X$ is compact. Theorem \ref{main1} provides some supporting examples for this question in the noncompact case.
\end{remark}

Let $p$ be a fixed point on a complete K\"ahler manifold $(X, \widetilde{g})$, and $d_{\widetilde{g}}$ the distance function from $p$. Given a real number $\alpha \geq 0$, we say that a holomorphic function $f \in \mathcal{O}(X)$ has \emph{polynomial growth of order at most $\alpha$} if there exists some constant $C(\alpha, f)$ so that
\[
|f(q)| \leq C(d(q, p)+1)^{\alpha}, \ \ \ \ \forall q \in X.
\]
For any real number $d \geq 0$, let $\mathcal{O}_d(X, \widetilde{g})$ denote the complex vector space of holomorphic functions with polynomial growth of order at most $d$. For the K\"ahler structures on $X$ with $BI \leq 0$ constructed in Theorem \ref{main1}, we may estimate the dimensions of holomorphic functions with polynomial growth.

\begin{corollary}\label{dimcountintro}
Given $(M, g)$ of complex dimension $n-r$ with $n>r \geq 1$ and a holomorphic line bundle $L \rightarrow M$ as in Theorem \ref{main1}, let $X$ be the total space of the vector bundle $E=L^{\oplus r}$ as in Theorem \ref{main1}. Then for any complete metric $(X, \widetilde{g})$ in the form of (\ref{themcalabi}) with $BI \leq 0$ and any $k \in \mathbb{Z}^{+}$, we have
\begin{equation}
\operatorname{dim} \mathcal{O}_k(X, \widetilde{g}) \geq \sum_{p=0}^{k} \Big[\operatorname{dim} H^0(M, L^{-p})  \binom{r+p-1}{r-1}\Big].   \label{dimcomp1}
\end{equation}
Moreover, if (\ref{dimcomp1}) holds along a sequence $\{k_j\}_{j=1}^{\infty} \rightarrow +\infty$, then the corresponding K\"ahler metric $\widetilde{g}$ has its $\widetilde{\chi}$ (in Theorem \ref{main1}) being identically zero, or equivalently $u$ (in Theorem \ref{main1}) a linear function.
\end{corollary}

From now on, we consider a wider class of K\"ahler metrics on the total spaces of certain line bundles. The goal is to study dimension estimates of holomorphic functions and related rigidity properties of such metrics. For the sake of convenience, we introduce two assumptions which will be frequently used in later sections.

\begin{itemize}
    \item  \namedlabel{AssumeA}{\textbf{Assumption (A)}}: Let $M$ be a compact K\"ahler manifold of complex dimension $n-1$ and $\pi: L \rightarrow M$ a holomorphic line bundle over $M$ with $c_1(L)<0$. Let $h$ denote a Hermitian metric on $L$, and $X$ the total space of $L$ with the corresponding complex structure $J$. We assume that $X$ admits a K\"ahler metric $\widetilde{g}$. Note that the zero section $M$ is a complex submanifold of $X$. Let $\widetilde{s}$ denote the distance from any point $x \in X$ to $M$ with respect to $\widetilde{g}$.  Let $Cut(M) \subset X$ denote the set of all cut-focal points with respect to $M$. We refer to Gray \cite[Chapter 8]{Gray} or Section \ref{sec4} for more background on normal exponential maps and related notions.

    \item  \namedlabel{AssumeB}{\textbf{Assumption (B)}}: Let $\sigma: \mathcal{N}_{M} \rightarrow M$ be the normal bundle with respect to $(M, g)$ in $X$. Assume that the corresponding normal exponential map $\exp^{\perp}: \mathcal{N}_{M} \rightarrow X$ is a smooth (fiberwise) diffeomorphism in the sense that it is a global diffeomorphism between $\mathcal{N}_{M}$ and $X$ and restricts to a diffeomorphism between $\sigma^{-1}(p)$ and $\pi^{-1}(p)$ for each $p \in M$. Note that $\widetilde{s} \in C^{\infty}(X \setminus M)$ in this case.

\end{itemize}

\begin{proposition} \label{upperdim}
Let $(X, \widetilde{g})$ be a complete K\"ahler manifold which satisfies \ref{AssumeA}. Suppose there exists some function $0 \leq k \in C^0[0, +\infty)$ with
\begin{equation}
    \int_0^{\infty} sk(s)ds<\infty
    \label{kscondi1}
\end{equation}
so that the holomorphic sectional curvature along $e_{n}=\frac{1}{\sqrt{2}}(\nabla \widetilde{s} -\sqrt{-1}J \nabla \widetilde{s})$ satisfies
\begin{equation}
R(e_n, \overline{e_n}, e_n, \overline{e_n}) \geq -k(\widetilde{s}(x)),\  \  \text{for any}\ \ x \in X \setminus (M \cup Cut(M)).        \label{curbound1}
\end{equation}
We define
\begin{equation}
\eta=e^{\int_0^{\infty} sk(s)ds}>0.      \label{eta_intro}
\end{equation}

Then for any $d \in \mathbb{R}^{+}$
\begin{equation}
\operatorname{dim} \mathcal{O}_d(X, \widetilde{g}) \leq \sum_{k=0}^{\lfloor \eta d \rfloor} \operatorname{dim} H^0(M, L^{-k}).   \label{dimcmain1}
\end{equation} Here $\lfloor c \rfloor$ denotes the greatest integer less than or equal to $c$. Moreover, if (\ref{dimcmain1}) holds as an equality along a sequence $\{d_j\}_{j=1}^{\infty} \subset \mathbb{R}^{+}$ which tends to infinity, then $k(\widetilde{s}) \equiv 0$ on $[0, \infty)$.
\end{proposition}

\begin{proposition} \label{lowerdim}
Assume that $(X, \widetilde{g})$ is a complete K\"ahler manifold which satisfies \ref{AssumeA} and \ref{AssumeB}. Suppose the holomorphic sectional curvature satisfies
\begin{equation}
-k(\widetilde{s}) \leq R(e_n, \overline{e_n}, e_n, \overline{e_n}) \leq 0,\ \  \  \text{for any}\ \ \widetilde{s} \in (0, \infty)      \label{curbound2}
\end{equation}
where $k(s) \geq 0$ is continuous on $[0, \infty)$ and satisfies (\ref{kscondi1}). Moreover, for each $p \in M$, $\pi^{-1}(p)$, the fiber of the line bundle $L$ at $p$, is totally geodesic in $(X, \widetilde{g})$. Then
\begin{equation}
\operatorname{dim} \mathcal{O}_d(X, \widetilde{g}) \geq \sum_{k=0}^{\lfloor d \rfloor} \operatorname{dim} H^0(M, L^{-k}),
\label{dimcmain2}
\end{equation}
holds for any $d \in \mathbb{R}^{+}$. Similarly, $k(\widetilde{s}) \equiv 0$ if (\ref{dimcmain2}) holds as an equality along a sequence $\{d_j\}_{j=1}^{\infty}$ tending to infinity.
\end{proposition}

The proofs of Proposition \ref{upperdim} and \ref{lowerdim} are based on several important works on complete K\"ahler manifolds with suitable curvature conditions. The curvature lower bound (\ref{curbound1}) with an integral condition (\ref{kscondi1}) is related to Greene-Wu's result \cite[Theorem C on p.56]{GW79} on quasi-isometry of exponential maps on Riemannian manifolds with a pole. Here we do prefer an integral bound than any pointwise one, as Theorem \ref{main1} indicates that we may perturb a negatively curved metrics so that its bisectional curvature tends to $-\infty$ along a sequence of points tending to infinity. The dimension upper bound (\ref{dimcmain2}) is proved by a three-circle type inequality first discovered by Liu \cite[Theorem 8 on p.2913]{Liu2016}, which was later generalized by Yu-Zhang \cite{YZ2022} where the integral condition (\ref{kscondi1}) was used again. When we apply their methods on $X$ in Proposition \ref{upperdim}, the Hessian estimate of the distance function to the zero section $M$ (see Tam-Yu \cite{TY2012}) plays an important role. Finally, the dimension lower estimate (\ref{dimcmain2}) relies on a crucial lemma (Lemma \ref{sup-inf}) due to Li-Tam \cite{LT1991} on the growth of the Green function on a complex plane. To address the equality case in (\ref{dimcmain2}), we have to make full use of Li-Tam's main result \cite[Corollary 3.3]{LT1991}. Note that the integral condition (\ref{kscondi1}) fits well with the finite total curvature condition on a complete surface, which is the main assumption of \cite{LT1991}.

It is worth pointing out that the lower bound (\ref{curbound2}) is necessary for the existence of holomorphic functions of polynomial growth. We do not need to assume (\ref{curbound2}) in Corollary \ref{dimcountintro} as K\"ahler metrics constructed in Theorem \ref{main1} satisfies some amount of symmetries along its fiber direction. In general, we may construct examples so that (\ref{dimcmain2}) fails if we only assume a nonpositive curvature condition, at least for the total space of a trivial line bundle.

\begin{example}\label{noholointro}
    For any compact K\"ahler manifold $(M, g)$ with $BI(g) \leq 0$, there exists a complete K\"ahler metric $\widehat{g}$ on $\mathbb{C}$ with nonpositive Gauss curvature so that any nonconstant holomorphic function on $X=M \times \mathbb{C}$ grows at least exponentially with respect to the product metric $\widetilde{g}=g \times \widehat{g}$.
\end{example}

Let $\mathcal{O}(X, Y)$ denote the space of holomorphic maps between two complex manifolds $X$ and $Y$. Our next result studies this space under a polynomial growth condition.

\begin{proposition} \label{xtoy}
Given two compact K\"ahler manifolds $M_1$ and $M_2$, assume that both $M_1$ and $M_2$ admit negative line bundles $L_1$ and $L_2$ respectively. Let $X$ and $Y$ be the total space of $\pi_1: L_1 \rightarrow M_1$ and $\pi_2: L_2 \rightarrow M_2$ respectively. We assume that $X$ admits a complete K\"ahler metric $\widetilde{g}$ which satisfies \ref{AssumeA} and \ref{AssumeB}, and $Y$ admits a K\"ahler metric $\widetilde{h}$ which satisfies \ref{AssumeA}. For any $x \in X$ and $y \in Y$, let $\s_1(x)=d_{\widetilde{g}} (f(x), M_1)$ and $\s_2(y)=d_{\widetilde{h}} (y, M_2)$. Moreover, we assume
the following conditions holds
\begin{enumerate}
\item  Each fiber of $L_1 \rightarrow M_1$ is totally geodesic.
\item  Let $e_{n}=\frac{1}{\sqrt{2}}(\nabla \widetilde{s_1} -J \nabla \widetilde{s_1})$ defined on $X \setminus M_1$. There exists some function $0 \leq k \in C^0[0, +\infty)$ with (\ref{kscondi1}) so that $R_{M_1}(e_n, \overline{e_n}, e_n, \overline{e_n}) \geq -k(\s_1(x))$ for any $x \in X \setminus M_1$.
\item  $BI(Y, \widetilde{h}) \leq 0$ and $M_2$ is Kobayashi hyperbolic.
\end{enumerate}
For any given $d>0$, we define
\begin{align}
\mathcal{O}_{d}(X, Y)=\{f \in \mathcal{O}(X, Y)\ |\  \s_2(f(x)) \leq  C(\s_1(x)+1)^d\ \text{for some constant}\  C\,\}. \label{mapOddef}
\end{align}
For any fixed $d < \frac{1}{\eta}$ where $\eta$ is defined in (\ref{eta_intro}), we get that $f$ is a constant map if we further assume that the image of $\pi_2 \circ f$ is a single point in $M_2$. In general, any $f \in \mathcal{O}_{d}(X, Y)$ factors through $\pi_1$, i.e. there exists $\underline{f} \in \mathcal{O}(M_1, Y)$ so that $f=\underline{f} \circ \pi_1$.
\end{proposition}

As a corollary of Proposition \ref{xtoy}, we consider the group of holomorphic biholomorphism of $X$ in Proposition \ref{lowerdim} ($\operatorname{Aut}(X)$ for short). Under the extra assumption that the base manifold $M$ is Kobayashi hyperbolic, $\operatorname{Aut}(X) \cap \mathcal{O}_{d}(X, X)$ is empty for any $d < \frac{1}{\eta}$. We observe that this result is sharp in view of K\"ahler metric $\widetilde{g}$ constructed in Theorem \ref{main1} with its $\widetilde{\chi}$ being identically zero, as $k(s) \equiv 0$ on this example. In the meantime, given any $f_{\diamond} \in Aut(M)$, we may use a global section of $L^{-1} \otimes {f_{\diamond}}^{\ast} L$ to construct a bundle isomorphism of $L$. All these bundle isomorphisms (including the identity map) lie in $\operatorname{Aut}(X) \cap \mathcal{O}_{1}(X, X)$. We refer to Liu \cite[Section 6 in the arXiv version]{Liu2016} and Ni \cite{Ni2021} for recent important results on Liouville theorems for holomorphic mappings under curvature assumptions.

\subsection{Further discussions}

We propose to study the function theory of nonpositively curved K\"ahler manifolds with Euclidean volume growth.

\begin{question}\label{Eucvol}
Let $(X, \widetilde{g})$ be a complete noncompact K\"ahler manifold with $BI \leq 0$. Assume that
$\widetilde{g}$ is of at most Euclidean volume growth in the sense that
\begin{align}
\limsup_{r \rightarrow \infty}\frac{\operatorname{Vol}(B(p, r))}{
r^{2n}} \in (0, +\infty).
\end{align} Here $B(p, r)$ denotes the geodesic ball with radius $r$ centered at some point $p \in X$.

(1) Is $X$ holomorphically convex?

(2) Does $X$ admit nonconstant holomorphic functions of polynomial growth?
\end{question}

\begin{remark}
Examples satisfying the assumption in Question \ref{Eucvol} can be found on total spaces of certain vector bundles as in Theorem \ref{main1} and on the Euclidean space $\mathbb{C}^n$ as shown in Appendix \ref{negUN}. Another interesting example is the smooth affine quadric
\begin{align*}
X_t=\{z_1^2+z_2^2+z_3^2=t\} \subset \mathbb{C}^3, \ \ t \neq 0.
\end{align*}
It is shown in \cite{YZ} that the induced Euclidean metric on $X_t$ satisfies $BI \leq 0$ and Euclidean volume growth. In particular, its bisectional curvatures have lots of zero directions, and the restriction of each coordinate function $z_i (1 \leq i \leq 3)$ on $\mathbb{C}^3$ is of linear growth.
\end{remark}

Siu-Yau \cite{SY} and Mok-Siu-Yau \cite{MSY} studied complete simply-connected noncompact
K\"ahler manifolds with nonpositive sectional curvature. Under a pointwise faster-than-quadratic decay assumption on scalar curvatures, they proved that these manifolds must be holomorphically isometric to the complex Euclidean space. We expect that the following question might be a suitable analogue of the Mok-Siu-Yau result on total space of vector bundles.

\begin{question}\label{decaynon}
Let $k(t)$ be any nonnegative continuous function on $[0, \infty)$ so that
\begin{equation}
    \int_0^{t} \sqrt{k(\tau)}d\tau =o(\ln t),\ \ \ \text{for}\ t \rightarrow \infty.
    \label{kscondi2}
\end{equation}
Does there exist a complete K\"ahler manifold $(X, \widetilde{g})$ satisfying \ref{AssumeA}
and $-k(\widetilde{s}) \leq BI(\widetilde{g}) \leq 0$?
\end{question}

We suspect that the answer to Question \ref{decaynon} is no. Note that (\ref{kscondi2}) is weaker than (\ref{kscondi1}). It might be plausible to develop some three-circle type inequality (as in Proposition \ref{thrcir}) to study function theory on $X$ in Question \ref{decaynon}. However, in view of examples constructed in Theorem \ref{main1}, it is the bisectional curvatures along the horizontal direction that play an essential role here.

\begin{remark}

Following the proof of Proposition \ref{upperdim}, we may derive a similar result for the total space of a vector bundle. Let $M$ be a compact K\"ahler manifold of complex dimension $n-r$ and $E$ be a holomorphic rank $r$ vector bundle over $M$. Assume $E$ is Griffiths negative. We use the same notation as in \ref{AssumeA}. Then the same conclusion as in Proposition \ref{upperdim} holds except that (\ref{dimcmain1}) is replaced by
\begin{equation}
\operatorname{dim} \mathcal{O}_d(X, \widetilde{g}) \leq \sum_{k=0}^{\lfloor \eta d \rfloor} \operatorname{dim} H^0(M, S^k(E^*)).   \label{dimcmain1v}
\end{equation} Here $S^k(E^{\ast})$ is the $k$-th symmetric product vector bundle of the dual bundle $E^{\ast}$. Similarly, we may propose Question \ref{decaynon} for the total spaces of a Griffiths negative bundle.
\end{remark}


The structure of the paper is the following. In Section \ref{sec2} we review Calabi's construction on a complete K\"ahler metric on the total space of a vector bundle. In Section \ref{sec3} we prove Theorem \ref{main1} and Corollary \ref{dimcountintro}. In Sections \ref{sec4} and \ref{sec5} we study polynomial growth holomorphic functions on (and mappings between) total spaces of certain line bundles, respectively. Propositions \ref{upperdim}, \ref{lowerdim}, and \ref{xtoy} are proved. In Appendix \ref{negUN} we give a systematic study of rotationally symmetric K\"ahler metrics on the complex Euclidean space. The results obtained in Appendix \ref{negUN} align with those in Theorem \ref{main1}. Finally, we exhibit examples of negatively curved metrics on complex planes without nonconstant polynomial growth holomorphic functions in Appendix \ref{noholoslow}, which complete the construction of Example \ref{noholointro}.

\vskip 0.2cm
\noindent\textbf{Acknowledgments.} Bo Yang would like to thank Fangyang Zheng for his generous help on applications of Theorem \ref{main1} and some results in Appendix \ref{negUN} in this paper. Bo Yang is also grateful to Xueyuan Wan for his interests and for bringing References \cite{KWZ2025} and \cite{Wan2024} to our attention.

\section{K\"ahler geometry on total spaces of vector bundles}\label{sec2}

Let $\pi: E \rightarrow M$ be a holomorphic vector bundle on a complex manifold $M$, where $(M, g)$ is K\"ahler and $h$ is a Hermitian metric on $E$. We call $X$ the total space of $E \rightarrow M$. We assume $\operatorname{dim}_{\mathbb{C}} M=n-r$ and $rank(E)=r$.

Let $z=(z_1, \cdots, z_{n-r})$ be a holomorphic chart on $U_{\alpha} \subset M$, and $\varphi: \pi^{-1}(U_{\alpha}) \rightarrow U \times \mathbb{C}$ be a local trivialization, and $e_{\mu}=\varphi^{-1}(z, \delta_{\mu})$ be a local basis of sections of $L$. Here $1 \leq \mu \leq r$ and $\delta_{\mu}$ represents the standard basis in $\mathbb{C}^r$ whose only nonzero entry is $1$ in the ${\mu}$-th position.

In Sections \ref{sec2} and \ref{sec3}, we use \textbf{the following convention}: the indices $i, j, k, l, p, q$ represent components from the base direction, and $\alpha, \beta, \gamma, \delta, \mu, \nu, \xi, \eta$ components from the fiber direction.

Recall the connection form and curvature form of $(E, h)$ is
\begin{align*}
& \theta_{\alpha}^{\beta}=\partial h_{\alpha \overline{\mu}} h^{\beta \overline{\mu}}, \ \  \ \Gamma_{\alpha i}^{\beta}=\frac{\partial h_{\alpha \overline{\mu}}}{\partial z_i} h^{\beta \overline{\mu}};\\
& \Theta_{\alpha}^{\beta}(\frac{\partial}{\partial z_i}, \frac{\partial}{\partial \overline{z_j}})=-\frac{\partial}{\partial \overline{z_j}} \Gamma_{\alpha i}^{\beta},\ \ \ R^{E}_{i \overline{j} \alpha \overline{\beta} }=\Theta_{\alpha}^{\gamma}(\frac{\partial}{\partial z_i}, \frac{\partial}{\partial \overline{z_j}}) h_{\gamma \overline{\beta}}.
\end{align*}

There is a natural complex structure on $X$ induced from $E \rightarrow M$. For any point $(p, v^{\mu} e_{\mu}) \in X$, we may write $T^{1, 0}X=\operatorname{span}\{\frac{\partial}{\partial z_i}, \frac{\partial}{\partial v^{\mu}}\}$ where $1 \leq i \leq n-r$ and $1 \leq {\mu} \leq r$.

Calabi (\cite{Calabi1}) considered the K\"ahler metric $\widetilde{g}$ on $X$ in the form of
\begin{align}
\widetilde{\omega}=\pi^{\ast}\omega_g+\sqrt{-1}\partial\overline{\partial} u.  \label{Cmetric}
\end{align}
Here for any point $(p, v^{\mu} e_{\mu}) \in X$, $u=u(t)$ is a suitable function of $t$ which is to be determined. Here $t=h(v^{\mu} e_{\mu}, \overline{v^{\mu} e_{\mu}})$. It was pointed out in {\cite[p.272]{Calabi1}} that each fiber is totally geodesic with respect to a K\"ahler metric in the form of (\ref{Cmetric}), see also \cite[Proposition 2.2 on p.2297]{HS2002}.

A natural question is whether the total space of $E \rightarrow M$ admits a complete K\"ahler metric. Making use of (\ref{Cmetric}), Calabi proved the following

\begin{proposition}[{\cite[Theorem 3.2 on p.275]{Calabi1}}]
Let $X$ be the total space of $E$ which is a holomorphic vector bundle over a complete K\"ahler manifold $(M, g)$. Assume that $E$ admits a Hermitian metric
$h$ whose Chern curvature is bounded from above, i.e. there exists some constant $C$ so that
\begin{align}
R^E(Z, \overline{W}, v^{\alpha} e_{\alpha}, \overline{w^{\beta} e_{\beta}}) \leq  C g(Z, \overline{W}) h(v^{\alpha} e_{\alpha}, \overline{w^{\beta} e_{\beta}}),  \label{ubound}
\end{align}
for any $p \in M$, $Z, W \in T_{p}^{1, 0}M$, and $v^{\alpha} e_{\alpha}, w^{\beta} e_{\beta} \in \pi^{-1}(p)$.
Then $X$ admits a complete K\"ahler metric in the form of (\ref{Cmetric}).
\end{proposition}

\begin{remark}
All the K\"ahler metrics in the form of (\ref{Cmetric}) admits a standard horizontal distribution. Namely the horizontal distribution is determined by the Chern connection of the vector bundle $(E, h)$. This point will be clear in (\ref{Cmetric2})
after we express (\ref{Cmetric}) in terms of local coordinates.
\end{remark}

Following \cite{Calabi1}, we set
\[
\nabla_{z_i}=\frac{\partial}{\partial z_i}-v^{\mu}\Gamma_{\mu i}^{\delta} \frac{\partial}{\partial v^{\delta}}, \ \ \ \nabla {v^{\mu}}=v^{\nu}\Gamma_{\nu i}^{\mu} dz^i+ d v^{\mu}.
\]
Next we choose a local $(1, 0)$ frame on $X$
\[
(\nabla_{z_1}, \cdots, \nabla_{z_{n-r}}, \frac{\partial}{\partial v^{1}}, \cdots, \frac{\partial}{\partial v^{r}}),
\]
and its coframe
\[
(dz^1, \cdots, dz^{n-r}, \nabla v^1, \cdots, \nabla v^r).
\]

Then (\ref{Cmetric}) can be rewritten (see \cite[p.274]{Calabi1})
\begin{align}
\widetilde{\omega}=\sqrt{-1}\Big[\Big(g_{i\bar{j}}-u'(t)R^E_{i \bar{j} \xi \bar{\nu}}v^{\xi}\overline{v^{\nu}}\Big) dz^i \wedge d\overline{z^j}+\Big(u^{\prime\prime}(t) h_{\mu\bar{\xi}} \overline{v^{\xi}} h_{\eta \bar{\nu}}v^{\eta}+u^{\prime}(t) h_{\mu \bar{\nu}}\Big)\nabla v^{\mu} \wedge \nabla \overline{v^{\nu}}\Big].  \label{Cmetric2}
\end{align}

\section{Negatively curved K\"ahler metrics on total spaces}\label{sec3}

In this section we prove Theorem \ref{main1}. In general, the curvature tensor of $\widetilde{\omega}$ depends on the background K\"ahler metric $g$ and the Hermitian metric $h$ on the fiber. The natural idea is to consider some special choices of Hermitian fiber structures $h$. Here we introduce two classes of line bundles and vector bundles respectively.

\begin{itemize}
   \item  \namedlabel{AssumeC}{\textbf{Assumption (C)}}: Let $(M, g)$ be a compact K\"ahler manifold with $BI(M, g)<0$, and there is a holomorphic line bundle $L \rightarrow M$ with a Hermitian metric $h$ whose curvature form is $-\lambda g$ for a constant $\lambda>0$.

    \item  \namedlabel{AssumeD}{\textbf{Assumption (D)}}: Let $(M, g)$ be a compact K\"ahler manifold with $BI(M, g)<0$, and there exists a holomorphic vector bundle $E$ of rank $r$ which is decomposable in the sense that $E=L^{\oplus r}$ for some $r>1$ where $L$ satisfies \ref{AssumeC}.
\end{itemize}

The general philosophy is that the induced metric along each fiber of $\widetilde{\omega}$ in the form of (\ref{Cmetric2}) is rotationally symmetric. Conversely, if we take the potential function which produces a $U(r)$-invariant K\"ahler metric on $\mathbb{C}^r$ with $BI<0$ as our choice of $u(t)$ in (\ref{Cmetric}), then it is possible to get a K\"ahler metric with $BI<0$ on the total space $X$. To motivate our discussion, we prove Theorem \ref{main1} in the line bundle case and the vector bundle case separately.

Fixing any point $p \in M$, we consider $(p, v^{\alpha}e_{\alpha})$ in $\pi^{-1}(p)$ and let $t=h(v^{\alpha}e_{\alpha}, \overline{v^{\alpha}e_{\alpha}})$. Without loss of generality, we may assume that at $p$, $g_{i\bar{j}}$ is diagonal with respect to $\{\frac{\partial}{\partial z_i}\}$ where $1 \leq i \leq n-r$, and so is $h_{\alpha\bar{\beta}}$ with respect to $\{e_{1}, \cdots, e_{r}\}$. Moreover, $dh_{\alpha \overline{\beta}}(p)=0$.

\begin{lemma}[curvature in the line bundle case] \label{lcurvature}
Let $L$ be a line bundle satisfying \ref{AssumeC}. We consider a local $(1, 0)$ frame
$\{\frac{\partial}{\partial z_1},\cdots, \frac{\partial}{\partial z_{n-1}}, \frac{\partial}{\partial v}\}$.
Then the local components of $\widetilde{g}$ is
\begin{align*}
\widetilde{g}_{i\bar{j}}&=Pg_{i\bar{j}}+Q t \frac{\partial \ln h}{\partial z_i} \frac{\partial \ln h}{\partial \overline{z_j}},\ \ 1 \leq i, j \leq n-1;\\
\widetilde{g}_{i\bar{v}}&=Q h v \frac{\partial \ln h}{\partial z_i}, \ \ \ \widetilde{g}_{v\bar{j}}=Q h \bar{v} \frac{\partial \ln h}{\partial \overline{z_j}},\ \ \widetilde{g}_{v\bar{v}}=Q h.
\end{align*} Here we set
\begin{align}
P=1+\lambda u'(t) t, \ \ \ Q=tu^{\prime\prime}(t)+u^{\prime}(t).  \label{MNdef}
\end{align}
Moreover, we solve the nonzero curvature components of $(X, \widetilde{g})$ at $p$:
\begin{align}
R_{i \bar{j}k \bar{l}}&=P R^M_{i\bar{j}k\bar{l}}-tQ\lambda^2 (g_{i\bar{j}}g_{k\bar{l}}+g_{i\bar{l}}g_{k\bar{j}}), \label{lC}\ \ \\
R_{v \bar{v}i \bar{j}}&=\Big(-tP_{tt}-P_t+\frac{(P_t)^2}{P}t\Big)hg_{i\bar{j}},   \ \ (1\leq i, j\leq n-1) \label{lB}\\
R_{v \bar{v}v \bar{v}}&=\Big(-tQ_{tt}+\frac{(Q_t)^2}{Q}t-Q_t\Big)h^2. \label{lA}
\end{align} For simplicity, $R^M_{i\bar{j}k\bar{l}}$ in (\ref{lC}) standard for the curvature component of $(M, g)$.
\end{lemma}

\begin{theorem}\label{lexistence}
Let $L$ be a line bundle satisfying \ref{AssumeC}. Then there is a natural correspondence between a smooth complete K\"ahler metric $\widetilde{g}$ in the form (\ref{Cmetric}) with $BI<0$ ($BI \leq 0$) on $X$ and a real-valued function $\chi(t) \in C^{\infty}[0, +\infty)$
satisfying $\chi(0)=0$ and $\chi^{\prime}(t)<0$ ($\chi^{\prime}(t) \leq 0$) for any $t \in [0, \infty)$. By the natural correspondence, we mean that it is bijective modulo a scaling.
\end{theorem}

\begin{proof}[Proof of Theorem \ref{lexistence}]

From Lemma \ref{lcurvature}, we see that any smooth complete K\"ahler metric in the form (\ref{Cmetric}) on $X$ is characterized by a single variable function $u(t)$ defined on $[0, \infty)$ such that
\begin{enumerate}
\item   $u(t) \in C^{\infty}[0, \infty)$ where we use one-sided derivatives at $t=0$.
\item   $tu'(t)$ is strictly increasing on $[0, \infty)$ and $u'(0)>0$.
\item   $\int_{0}^{\infty} \sqrt{\dfrac{tu^{\prime\prime}(t)+u^{\prime}(t)}{t}}dt=\infty$.
\end{enumerate}

Note that the condition (iii) ensures the completeness of the metric as each fiber of $L \rightarrow M$ is totally geodesic.

Next we study the bisectional curvature of $\widetilde{g}$. Consider
\begin{align*}
\widetilde{U}=\sum_{i=1}^{n-1} u^i {\frac{\partial}{\partial z^i}}+\varepsilon  \frac{\partial}{\partial v}, \  \ U=d\pi (\widetilde{U});\ \ \ \
\widetilde{W}=\sum_{i=1}^{n-1} w^i {\frac{\partial}{\partial z^i}}+\mu  \frac{\partial}{\partial v}, \  \ W=d\pi (\widetilde{W}).
\end{align*}
Using (\ref{lC}), (\ref{lB}), and (\ref{lA}), we have
\begin{align}
&R(\widetilde{U}, \overline{\widetilde{U}}, \widetilde{W}, \overline{\widetilde{W}})  \label{anyBI}\\
=& P R^{M}(U, \overline{U}, W, \overline{W})-tQ \lambda^2 [g(U, \overline{U})g(W, \overline{W})+g(U, \overline{W})g(W, \overline{U})]  \nonumber\\
& + II \, Q h \, g(\mu U+\varepsilon V, \overline{\mu U+\varepsilon V})+|\varepsilon|^2 |\mu|^2 \, I\, Q^2h^2,   \nonumber
\end{align}
where we set
\begin{align}
I=\frac{1}{Q^2}(-tQ_{tt}+\frac{(Q_t)^2}{Q}t-Q_t),\ \ \ II=\frac{1}{Q}(-tP_{tt}-P_t+\frac{(P_t)^2}{P}t).  \label{IIIdef}
\end{align}
Therefore we conclude that $\widetilde{g}$ satisfies $BI<0$ if and only if both $I$ and $II$ are negative for all $t \in [0, \infty)$.

We introduce a function $\chi: [0, +\infty) \rightarrow \mathbb{R}$ by
\begin{align}
\chi(t)=-\frac{t Q_t}{Q}.   \label{chidef}
\end{align}
Next we observe that
\[
I=\frac{\chi^{\prime}(t)}{Q}, \ \ \ II=\frac{\lambda (1-\chi)}{P}\Big(\frac{\lambda t Q}{1-\chi}-P\Big).
\]
Motivated by \cite[Theorem 2 on p.525]{WZ} (see also Proposition \ref{propA2}), we will show that any strictly decreasing function $\chi \in C^{\infty}[0, \infty)$ with $\chi(0)=0$ and $\chi^{\prime}(0)<0$ corresponds to a complete K\"ahler metric in the form (\ref{Cmetric}) with $BI<0$. Once we have such a $\chi(t)$ and $Q(0)>0$ prescribed, we may solve
\begin{align}
Q(t)=Q(0)\exp(-\int_0^t \frac{\chi(\tau)}{\tau}d\tau), \ \ u^{\prime}(t)=\frac{1}{t} \int_0^t Q(\tau)d\tau, \ \ u(t)=\int_0^t u^{\prime}(\tau)d\tau.  \label{Ntaueqn}
\end{align}
It is straightforward to check $u(t)$ can be used to produce a smooth complete K\"ahler metric, i.e. it satisfies the conditions (i), (ii), and (iii) in the above. It only remains to check
that $II$ defined in (\ref{IIIdef}) (hence $\frac{\lambda t Q}{1-\chi}-P$) is negative on $[0, \infty)$. To that end, it suffices to note
\[
\frac{d}{dt}\Big(\frac{\lambda t Q}{1-\chi}-P\Big)=\lambda t Q\frac{\chi^{\prime}(t)}{(1-\chi)^2} <0,\ \ \ \ \  (\frac{\lambda t Q}{1-\chi}-P\Big)\Big|_{t=0}=-1.
\]
\end{proof}

\begin{corollary}\label{bigeuclidean}
Given any complete K\"ahler metric $\widetilde{g}$ with $BI \leq 0$ from the correspondence stated in Theorem \ref{lexistence}, there exists some constant $C_1>0$ so that
\begin{align}
\operatorname{Vol}(B(p, s)) \geq C_1 s^{2n}, \ \ \forall s \geq 0.  \label{bigger}
\end{align} Here $p$ is a point on the zero section $M \subset X$, and $B(p, s)$ the geodesic ball with respect to $\widetilde{g}$.
\end{corollary}

\begin{proof}[Proof of Corollary \ref{bigeuclidean}]
Let $T(M, c)$ denote the `tube' set which consists of all points $q$ so that $d(q, M) \doteq \min_{p \in M} d(q, p) \leq c$.

From Lemma \ref{lcurvature}, we may solve the volume form of $\widetilde{g}$ by a direct linear algebra argument
\begin{align}
\frac{{\widetilde{\omega}}^n}{n!}=P^{n-1}Q h \,  (\sqrt{-1})^{n-1} \det(g_{i\bar{j}})\,\prod_{p=1}^{n-1}{dz_p \wedge d\overline{z_p}} \wedge \sqrt{-1} dv \wedge d\bar{v}.  \label{volform}
\end{align}
Note that (\ref{volform}) not only holds along the fiber $\pi^{-1}(p)$ where we assume $dh_{\alpha \overline{\beta}}(p)=0$, it also holds in $\pi^{-1}(U)$ where $U \subset M$ is a open neighborhood of $p$.

Recall that both the base $M$ and each fiber of $L \rightarrow M$ is totally geodesic with respect to $(X, \widetilde{g})$. Moreover, $d(q, M)=d(q, p)$ where $p=\pi(q)$. Let $\widetilde{s}(q)=d(q, M)$. For a fixed reference point $p \in M$, $\widetilde{s}(q)$ is equivalent to $d(q, p)$ when $\widetilde{s}$ is sufficiently large. Therefore, we only need to prove (\ref{bigger}) for $\operatorname{Vol}(T(M, \widetilde{s}))$ instead. To that end, we compare
\[
\operatorname{Vol}(T(M, \widetilde{s}))=\pi \int_{0}^t P^{n-1} Q d\tau \, \operatorname{Vol}(M, g)=\frac{\pi (P^n-1)}{n \lambda}\operatorname{Vol}(M, g), \ \ \ \widetilde{s} \doteq \int_0^t \frac{\sqrt{Q}}{2\sqrt{\tau}} d\tau.
\]
In the meantime, it follows from (\ref{MNdef}) and (\ref{chidef}) that
\[
\frac{d}{d\widetilde{s}} P=2\lambda \sqrt{tQ}, \ \ \ \ \frac{d^2}{d\widetilde{s}^2}P=4\lambda(1-\chi) \geq 4\lambda.
\]
After integration by twice, we get $P \geq 2\lambda \widetilde{s}^2$, which proves (\ref{bigger}).
\end{proof}

\begin{example}\label{euclideanvol}
Given any $n \geq 2$ and $(M, g)$ to be an $(n-1)$-dimensional compact ball quotient with its standard K\"ahler-Einstein metric $g$ so that $Ric(g)=-g$, by Theorem \ref{lexistence} we may construct a complete K\"ahler metric $\widetilde{g}$ with $BI<0$ on the total space $X$ of the anticanonical line bundle $K_M^{-1}$. Indeed, if we choose $\chi(t)$ so that $\chi(\infty)$ is finite, then by Corollary \ref{bigeuclidean} the corresponding metric $\widetilde{g}$ has Euclidean volume growth in the sense that
\begin{align}
\lim_{s \rightarrow \infty}\frac{\operatorname{Vol}(B(p, s))}{
s^{2n}}=\frac{[2(1-\chi(\infty))]^n}{n}\operatorname{Vol}(M, g).
\end{align}
Moreover, if we choose $\chi'(0)$ and $Q(0)$ so that
$\frac{\chi'(0)}{Q(0)}-\lambda^2<0$, then we may show that $\widetilde{g}$ does not have negative sectional curvature along the zero section, see also Theorem \ref{mainApp}.
\end{example}

The following result is a more precise statement of the line bundle case of Corollary \ref{dimcount1}.

\begin{corollary}\label{dimcount1}
Given $(M, g)$ of complex dimension $n-1$ and a holomorphic line bundle $L \rightarrow M$ as in Theorem \ref{main1}, we consider $X$ the total space of the line bundle $L$ (i.e. we pick $E=L$ in Theorem \ref{main1}). Then for any complete metric $(X, \widetilde{g})$ in the form of (\ref{themcalabi}) with $BI \leq 0$ and any $k \in \mathbb{Z}^{+}$, we have
\begin{equation}
\operatorname{dim} \mathcal{O}_k(X, \widetilde{g}) \geq \sum_{p=0}^{k} \operatorname{dim} H^0(M, L^{-p}).   \label{dimcomp11}
\end{equation}
Moreover, if (\ref{dimcomp11}) holds along a sequence $\{k_j\}_{j=1}^{\infty} \rightarrow +\infty$, then the corresponding K\"ahler metric $\widetilde{g}$ has its $\chi$ (in Theorem \ref{lexistence}) being identically zero, or equivalently $u$ (in Theorem \ref{main1}) linear.
\end{corollary}

\begin{proof}[Proof of Corollary \ref{dimcount1}]

Let $\{U_i\}$ be an open cover of $M$, $\psi_{i,k}: \pi_1^{-1}(U_i) \rightarrow U_i \times \mathbb{C}$, and $\varphi_i: \pi_2^{-1}(U_i) \rightarrow U_i \times \mathbb{C}$ be local trivializations of $\pi_1: L^{-k} \rightarrow M$ and $\pi_2: L \rightarrow M$ respectively. For simplicity, let $p_2$ denote the projection to the second factor in either of $\psi_{i,k}$ or $\varphi_i$.

For any given $s \in H^0(M, L^{-k})$, we observe
\[
f_k(q)=p_2 \circ \psi_{i,k} \circ s(\pi_2(q)) \cdot \Big(p_2 \circ \varphi_i (q)\Big)^k, \ \ q \in \pi_2^{-1}(U_i) \subset X.
\]
is a global holomorphic function on $X$. Note that the local coordinate $p_2 \circ \varphi_i (q)$ is exactly $v$ in Lemma \ref{lcurvature}.

Conversely, any global holomorphic function $f$ on $X$ has a Taylor expansion on $\pi_2^{-1}(U_i)$ as
\begin{align}
f(q)=\sum_{k=0}\frac{1}{k!} f_k(\pi_2(q)) \Big(p_2 \circ \varphi_i (q)\Big)^k   \label{taylor}
\end{align} where each $f_k$ defines an element in $H^0(M, L^{-k})$. This point can be checked by comparing transition functions along intersections of $U_i \cap U_j$.

As $\widetilde{g}$ in Theorem \ref{lexistence}, recall the distance $\widetilde{s}(q)$ of any point $q \in X$ to the zero section $M$ is $\int_0^{t} \frac{\sqrt{Q}}{2\sqrt{\tau}}d\tau$. In view of (\ref{Ntaueqn}), we get $Q(t)$ is increasing and $Q(t) \geq u'(0)>0$. Hence
\[
\widetilde{s}(q) \geq \sqrt{u'(0)}\sqrt{t}=\sqrt{u'(0)} \sqrt{h} |p_2 \circ \varphi_i(q)|.
\]
For any fixed reference point $p \in M$, $d(q, p)$ is equivalent to $\widetilde{s}(q)$. We see that $(p_2 \circ \varphi_i(p))^k$ is of polynomial growth with order at most $k$. By (\ref{taylor}), we get the dimension comparison inequality (\ref{dimcomp11}).

Now we turn to the equality case. If (\ref{dimcomp11}) holds along a sequence $\{k_j\}_{j=1}^{\infty} \rightarrow +\infty$, we first claim
\begin{align}
\limsup_{\widetilde{s} \rightarrow \infty} \frac{\sqrt{t}}{\widetilde{s}}>0  \label{limit1}.
\end{align}
If (\ref{limit1}) holds, recall
\[
\frac{\frac{d\sqrt{t}}{dt}}{\frac{d\widetilde{s}}{dt}}=
\frac{\frac{1}{2\sqrt{t}}}{\frac{\sqrt{Q(t)}}{2\sqrt{t}}}
=\frac{1}{\sqrt{Q(t)}}.
\]
We see $\lim_{t \rightarrow \infty} Q(t)=Q(\infty)$ is finite. It follows from (\ref{Ntaueqn}) that $\chi$ is identically zero as $-\chi$ is increasing. If so, $u$ is linear from (\ref{Ntaueqn}).

It remain to show (\ref{limit1}). Assume the contrary, then $\lim_{t \rightarrow \infty} Q(t)=\infty$. By (\ref{Ntaueqn}) that $-\chi(t_0)>0$ for some $t_0>0$. As $-\chi$ increasing, we get $\int_0^t \frac{\chi(\tau)}{\tau}d\tau \geq -\chi(t_0) \ln (\frac{t}{t_0})$ for any $t>t_0$. Using (\ref{Ntaueqn}) again, we get $Q(t) \geq c_2 t^{c_1}$ for some positive $c_1, c_2$. It further leads to $\widetilde{s} \geq c_3 t^{\frac{c_1+1}{2}}$. In other words, we get $\sqrt{t} \leq c_4 {\widetilde{s}}^{\alpha}$ where $\alpha=\frac{1}{1+c_1}$. Then we have
\[
(\sqrt{t})^{k+1} \leq {\widetilde{s}}^{\alpha (k+1)} \leq {\widetilde{s}}^{k}
\] if we pick $k$ large enough so that $\alpha<1-\frac{1}{k+1}$. It implies that for any such $k$
\begin{equation}
\operatorname{dim} \mathcal{O}_{k}(X, \widetilde{g}) \geq \sum_{p=0}^{k+1} \operatorname{dim} H^0(M, L^{-p}). \label{dimcomp12}
\end{equation} Recall we assume that (\ref{dimcomp11}) holds along an infinite sequence $\{k_j\}$. We get $H^0(M, L^{-k_j-1})=\{0\}$. It is impossible as $L^{-1}$ is ample.

\end{proof}

Now we turn to the vector bundle case. If $E$ is a vector bundle satisfying \ref{AssumeD}, we may endow $E$ with a Hermitian metric as a direct sum of $r$ copies of $h$. Then the curvature of $E$ satisfies
\[
R^{E}_{\alpha \overline{\beta} i \overline{j}} \doteq R^{E} (\frac{\partial}{\partial z_i}, \frac{\partial}{\partial \overline{ z_j}}, e_{\alpha}, \overline{e_{\beta}})=Ric^{L}_{i\bar{j}}h\, \delta_{\alpha\bar{\beta}}.
\]
Now we consider $(p, v^{\alpha}e_{\alpha})$ in $\pi^{-1}(p)$, where \textbf{$v^{\alpha}=0$} for $1 \leq \alpha \leq r-1$.

\begin{lemma}\label{vmetric}
Let $E$ be a vector bundle satisfying \ref{AssumeD}.  Then the local components of $\widetilde{g}$ with respect to the local holomorphic frame $\{\frac{\partial}{\partial z_1},\cdots, \frac{\partial}{\partial z_{n-r}}, \frac{\partial}{\partial v^1}, \cdots,  \frac{\partial}{\partial v^r}\}$ are
\begin{align*}
\widetilde{g}_{i\bar{j}}&=P\,g_{i\bar{j}}+Q_{\xi \bar{\eta}} v^{\xi}\overline{v^{\eta}} h \frac{\partial \ln h}{\partial z_i} \frac{\partial \ln h}{\partial \overline{z_j}},\ \ 1 \leq i, j \leq n-r;\\
\widetilde{g}_{i\bar{\beta}}&=Q_{\xi\bar{\beta}} v^{\xi} h \frac{\partial \ln h}{\partial z_i}, \ \ \ \widetilde{g}_{\alpha \bar{j}}=Q_{\alpha \bar{\eta}} h \bar{v^{\eta}} \frac{\partial \ln h}{\partial \overline{z_j}},\\
\widetilde{g}_{\alpha\bar{\beta}}&=Q_{\alpha\bar{\beta}} h.
\end{align*} Here we set
\begin{align}
P=1+\lambda u'(t) t, \ \ \ Q_{\alpha\bar{\beta}}=u^{\prime\prime}(t) h \overline{v^{\alpha}}v^{\beta}+u^{\prime}(t)\delta_{\alpha\bar{\beta}}.
\end{align}
In particular, at the point $(p, v^{\alpha}e_{\alpha})$, $Q_{\alpha\bar{\beta}}$ becomes diagonal with elements $(S, \cdots, S, T)$ where
\begin{align}
S=u^{\prime}(t),\ \ \ T=tu^{\prime\prime}(t)+u^{\prime}(t).  \label{PQdef}
\end{align}
\end{lemma}

\begin{lemma}[curvature in the vector bundle case] \label{vcurvature}
Under the same assumption in Lemma \ref{vmetric}, the nonzero curvature components of the K\"ahler metric $\widetilde{g}$ at $(p, v^{\alpha}e_{\alpha})$ are:
\begin{align}
R_{i \bar{j}k \bar{l}}&=P R_{i\bar{j}k\bar{l}}^{M}-tT\lambda^2 (g_{i\bar{l}}g_{k\bar{j}}+g_{i\bar{j}}g_{k\bar{l}}), \label{vC2}\\
R_{k \bar{j}r \bar{r}}&=\Big(-tT_t-T +\lambda t\frac{T^2}{P} \Big)h\lambda g_{k\bar{j}}, \label{vB2}\\
R_{r \bar{r}r \bar{r}}&=-h^2\Big[\Big(2u^{\prime\prime}(t)+4tu^{(3)}(t)+t^2u^{(4)}(t)\Big)-\frac{(T_t)^2t}{T}\Big],   \ \  \label{vA}\\
R_{\mu \bar{\mu}r \bar{r}}&=-h^2 [u^{\prime\prime}(t)+ tu^{(3)}(t)],   \ \  \label{vB1}\\
R_{\mu \bar{\mu}\mu \bar{\mu}}&=2R_{\mu \bar{\mu}\alpha \bar{\alpha}}=-2u^{\prime\prime}(t) h^2  \  (\alpha \neq \mu), \ \  \label{vC1}\\
R_{k \bar{j}\mu \bar{\mu}}&=-hT\lambda g_{k\bar{j}}. \label{vD}
\end{align}
In the above we always assume $1\leq i, j, k, l\leq n-r$, $1 \leq \alpha, \beta, \mu, \nu \leq r-1$, and $r$ means $\frac{\partial}{\partial v^r}$.
\end{lemma}

\begin{theorem}\label{vexistence}
Let $E$ be a vector bundle satisfying \ref{AssumeD} as in Lemma \ref{vmetric}. Then there is a natural correspondence between a smooth complete K\"ahler metric $\widetilde{g}$ in the form (\ref{Cmetric}) with $BI<0$ on $X$ and a strictly decreasing function $\widetilde{\chi}(t) \in C^{\infty}[0, +\infty)$ satisfying $\widetilde{\chi}(0)=0$ and $\widetilde{\chi}^{\prime}(0)<0$. Similar results also hold for $BI \leq 0$ as in Theorem \ref{lexistence}.
\end{theorem}

\begin{proof}[Proof of Theorem \ref{vexistence}]

The proof is similar to that of Theorem \ref{lexistence}. From Lemma \ref{vmetric}, we see that any smooth complete K\"ahler metric in the form (\ref{Cmetric}) is characterized by a single variable function $u(t)$ defined on $[0, \infty)$ such that
\begin{enumerate}
\item   $u(t) \in C^{\infty}[0, \infty)$ where we use one-sided derivatives at $t=0$.
\item   Both $S(t)$ and $T(t)$ defined in (\ref{PQdef}) are strictly positive on $[0, \infty)$.
\item   $\int_{0}^{\infty} \sqrt{\frac{T(t)}{t}}dt=\infty$.
\end{enumerate}

From Lemma \ref{vcurvature} and the proof of Theorem \ref{lexistence}, if such a complete K\"ahler metric satisfies $BI<0$, it follows that $\widetilde{\chi}(t)=-\frac{t T_t}{T}$ is strictly decreasing on $[0, \infty)$. Moreover, both $u^{\prime\prime}(t)$ and $(tu^{\prime\prime})^{\prime}$ are strictly positive on $[0, \infty)$.
Now we show a somewhat surprising fact that the latter condition automatically holds true if $\widetilde{\chi}(0)<0$ and $\widetilde{\chi}'(t)<0$ on $[0, \infty)$. Recall
\begin{align*}
& T(t)=T(0)\exp(-\int_0^t \frac{\widetilde{\chi}(\tau)}{\tau}d\tau),   \ \ \ u'(t)=\frac{1}{t}\int_0^t T(\tau)d\tau.\\
& u^{\prime\prime}(t)=-\frac{1}{t^2}\int_0^t T(\tau)d\tau+\frac{T(t)}{t}=\frac{1}{t^2}\int_0^t \Big(T(t)-T(\tau)\Big)d\tau.
\end{align*}
Since $\widetilde{\chi}(t)<0$, $T(t)$ is strictly increasing and we get $u^{\prime\prime}(t)>0$ when $t>0$. At $t=0$, the same holds once we note $\widetilde{\chi}^{\prime}(0)<0$. It remains to check that $tu^{\prime\prime}(t)$ is strictly increasing for any $t \geq 0$.

Note that the derivative of $tu^{\prime\prime}(t)$ at $t=0$ is $u^{\prime\prime}(0)>0$. It suffices to consider $t>0$.
\begin{align*}
(tu^{\prime\prime})^{\prime}&=\Big(\frac{\int_0^t (T(t)-T(\tau))d\tau}{t}\Big)^{\prime}\\
&=\frac{1}{t^2}\Big(t^2T'(t)-tT(t)+\int_0^t T(\tau)d\tau\Big)
\end{align*}
Consider the numerator of the above
\begin{align*}
Num(t)=t^2T'-tT+\int_0^t T(s)ds=tT(t)(-\widetilde{\chi}-1)+\int_0^t T(s)ds.
\end{align*}
Obviously $Num(0)=0$ and
\begin{align*}
Num'(t)&=(tT)'[-\widetilde{\chi}-1]+t (-\widetilde{\chi})'T(t)+T(t)\\
&=T(t)\Big[ (\widetilde{\chi}+1)^2+t(-\widetilde{\chi})' +\frac{3}{4}\Big]>0.
\end{align*}

Now we may conclude that any smooth function $\widetilde{\chi}$ on $[0, \infty)$ with $\widetilde{\chi}(0)<0$ and $\widetilde{\chi}'(t)<0$ corresponds to a complete K\"ahler metric with $BI \leq 0$. Indeed, we consider two $(1, 0)$ tangent vectors at $(p, v^{\alpha}e_{\alpha})$ in the form
\begin{align*}
\widetilde{U}&=\sum_{i=1}^{n-1} u^i {\frac{\partial}{\partial z^i}}+\sum_{\xi=1}^{r-1} a^{\xi}\frac{\partial}{\partial v^{\xi}}+\varepsilon  \frac{\partial}{\partial v^{r}}, \  \ U=d\pi (\widetilde{U});\\
\widetilde{W}&=\sum_{i=1}^{n-1} w^i {\frac{\partial}{\partial z^i}}+\sum_{\eta=1}^{r-1} b^{\eta}\frac{\partial}{\partial v^{\eta}}+\mu  \frac{\partial}{\partial v^{r}}, \  \ W=d\pi (\widetilde{W}).
\end{align*}
Following (\ref{anyBI}) in the proof of Theorem \ref{lexistence}, we may expand $R(\widetilde{U}, \overline{\widetilde{U}}, \widetilde{W}, \overline{\widetilde{W}})$ to show that it is indeed negative as long as neither $\widetilde{U}$ nor $\widetilde{W}$ is zero.

\end{proof}

Let $\mathcal{P}_k$  be homogeneous polynomials on $\mathbb{C}^n$ with complex coefficients and degree $k$. Note that $\mathcal{P}_k$ is a (complex) linear space after adding $0$. Moreover, $\operatorname{dim}_{\mathbb{C}}(\mathcal{P}_k \cup \{0\})=\binom{n+k-1}{n-1}$. In view of this fact, Corollary \ref{dimcountintro} follows from the proof of Corollary \ref{dimcount1}.

\section{Dimension estimates on holomorphic functions}\label{sec4}

\subsection{Dimension estimates: upper bounds}\label{uppersub}

In this subsection, we consider a class of K\"ahler metrics on total space of line bundles, and bound dimensions of holomorphic functions with polynomial growth from the above.

\subsubsection{Preliminary results on Hessian comparison}

We begin with a lemma on an ODE estimate.

\begin{lemma}[based on {\cite[Lemma 4.5 on p.62]{GW79}}]\label{eta}
Let $0 \leq k(t) \in C^0[0,+\infty)$ and $u(t) \in C^2[0, \infty)$ the solution to the following equation
\begin{equation}
u^{\prime\prime}(t)-k(t)u(t)=0, \ \ \text{with}\ \ u(0)=0,\ u'(0)=1.     \label{res}
\end{equation}
If $\int_{0}^{+\infty}sk(s)\,ds < \infty$, then $1\le u'(t)\le \eta$ where
\begin{equation}
\eta=\exp(\int_{0}^{+\infty}sk(s)\,ds).  \label{etadef}
\end{equation}
Moreover, $\lim_{t \to +\infty} u'(t) \le \eta$ and the equality holds if and only if $\eta=1$ and $k(t)\equiv 0$.
\end{lemma}
\begin{proof}[Proof of Lemma \ref{eta}]
It follows from the initial condition in (\ref{res}) that $u''=ku\ge 0$ for any $t \geq 0$. Hence $u'(t) \geq u'(0)=1$. Now we solve
\[
(\frac{u}{u'})'=\frac{(u')^2-u''u}{(u')^2}=1-k\frac{u^2}{(u')^2} \le 1.
\]
Hence
\begin{equation*}
    \frac{u''}{u'}=k\frac{u}{u'}\le tk(t).
\end{equation*}
By integration we have
\begin{equation*}
    \log u'(t)=\int_{0}^{t}\frac{u''(s)}{u'(s)}\,ds \le \int_{0}^{+\infty}sk(s)\,ds < \infty.
\end{equation*}
It follows that $\lim_{t \to +\infty} u'(t)$ exists and $u'< \eta$ as $u'$ is increasing.

If $k(t_0)>0$ for some $t_0 >0$, there exists small $\delta>0$ such that $\frac{u^{\prime\prime}(t)}{u'(t)}< tk(t)-\epsilon$ for some $\epsilon>0$ on $[t_0,t_0+\delta]$.
It follows that
\[
\ln u'(t) \leq \int_{0}^{+\infty}sk(s)\,ds-\epsilon \, \delta < \int_{0}^{+\infty}sk(s)\,ds.
\]
Therefore $\lim_{t \to +\infty} u'(t)<\eta$ whenever $k(t)$ is not identically zero.
\end{proof}

Let $(\M, \g)$ be a complete noncompact Riemannian manifold and $M \subset \widetilde{M}$ be a compact submanifold with induced metric $g$. Let $\s$ be the distance function of any point $q \in \M$ away from $M$.
As in the case of $M$ being a single point, we have the notion of normal geodesic maps. Similarly, we may define
focal points, Fermi coordinates, and cut-focal points, compared with conjugate points, geodesic normal coordinates, and cut points respectively. Let $Cut(M) \subset \M$ denote the set of all cut-focal points with respect to $M$. Our goal is to estimate the Hessian of $\s$ along the radial direction away from $Cut(M)$. We refer to Gray \cite{Gray} and Sakai \cite{Sakai} for more background on normal exponential maps, cut-focus points, and related notions. In fact, we are mainly concerned with the K\"ahler case. We refer to Li-Wang \cite{LW2005}, Tam-Yu \cite{TY2012} and Ni-Zheng \cite{NZ2018} for more information on comparison theorems in K\"ahler geometry. In particular, the latter two discuss in greater detail comparison results related to the distance function away from a compact submanifold in a K\"ahler manifold.

We recall the Riccati comparison principle.

\begin{lemma}[{\cite[Proposition 6.4.1 on p.254]{Pe}}]\label{rct}
    Let $f_1,f_2:(0,+\infty) \to\R$ be differentiable functions. Suppose  $f_1,f_2$ satisfy that
    \begin{equation}
        f'_1+(f_1)^2 \le f'_2+(f_2)^2,
    \end{equation}
    and
    \begin{equation}
        \limsup_{t\to 0+}(f_2(t)-f_1(t)) \ge 0.
    \end{equation}
    Then $f_1(t) \le f_2(t)$.

\end{lemma}
\begin{proof}[Proof of Lemma \ref{rct}]
    Let $F(t)=\int (f_1 +f_2) \,dt$.
    $$\frac{d}{dt}(f_2-f_1)e^F=(f'_2-f'_1+f_2^2-f_1^2)e^F \ge 0.$$
    Hence $(f_2-f_1)e^F$ is increasing and $(f_2-f_1)e^F\ge 0$.
\end{proof}

Write $\mathcal{N}=\nabla\s$. Define $H: T\M \rightarrow T\M$ by
\begin{equation}
\g(H(X),Y)=\nabla^2\s(X,Y).      \label{Hmapdef}
\end{equation}
Since
\[
\nabla^2\s(X,Y)=XY(\s)-\nabla_X Y(\s)=\g(\nabla_X \mathcal{N},Y),
\]
we have $H(X)=\nabla_X \mathcal{N}$, which is the opposite of the shape operator defined in \cite[p.33]{Gray}.

Define $R_{\mathcal{N}}: T\M \rightarrow T\M$ by
\[
R_{\mathcal{N}}(X) = R(X,\mathcal{N})\mathcal{N}.
\]
In this section the curvature tensor of $\g$ is denoted by $R$.
\begin{proposition}[{\cite[Lemma 3.2 on p.34]{Gray}}]\label{hj}
$H$ and $R_{\mathcal{N}}$ satisfies the following Riccati equation
\begin{align}
\nabla_{\mathcal{N}} H +H^2 +R_{\mathcal{N}}=0.
 \label{hj1}
\end{align}
\end{proposition}

Now we consider $(\M^{n},J,\g)$ is a complete K\"ahler manifold with complex dimension $n$ and $M^{n-r} \subset \widetilde{M}$ be a compact complex submanifold with the induced metric $g$. Given any point $p \in M$, let $\gamma(t): [0,b] \to \widetilde{M}$ be a unit-speed geodesic with $\gamma(0)=p$ and $\gamma^{\prime}(t)=\mathcal{N}$. We extend $H$ as a self-adjoint endomorphism $T^{\mathbb{C}} \M \rightarrow T^{\mathbb{C}}\M$. Now we choose a parallel unitary frame $\{e_1,\cdots,e_{n-r},e_{n-r+1},\cdots,e_{n}\}$ of $T^{1, 0} \M$ along $\gamma(t)$ such that $e_n=\frac{1}{\sqrt 2}(\mathcal{N} -\sqrt{-1}J\mathcal{N})$ and $\{e_{1}(0),\cdots,e_{n-r}(0)\} \in T_p^{1,0} M$. Now we are ready to rewrite (\ref{hj1}) in terms of $\{e_1,\cdots,e_{n}\}$.

Let $H(e_k)=A_{ik}e_i+B_{ik}\overline {e_i}$. It follows from (\ref{Hmapdef}) that
\[
\s_{kl}=g(H(e_k),e_l), \quad \s_{k\bar l}=g(H(e_k),\overline{e_l} ).
\]
We have
\begin{equation}
H(e_k)=\s_{k \bar i}e_i+\s_{ki}\overline{e_i}.
 \label{2e}
\end{equation}

Since the frame is parallel, $\nabla_{\mathcal{N}} e_i=0$ for all $i$.
\begin{equation}
\begin{aligned}
(\nabla_{\mathcal{N}} H)(e_k)&=\nabla_{\mathcal{N}} (H(e_k))-H(\nabla_{\mathcal{N}} e_k)=\nabla_{\mathcal{N}} (\s_{k \bar i}e_i+\s_{ki}\overline {e_i})\\
&= \partial_t (\s_{k \bar i})e_i+\partial_t(\s_{ki})\overline {e_i}.
\end{aligned}\label{wd}
\end{equation}
\begin{equation}
\begin{aligned}
H^2(e_k)&=H(\s_{k \bar i}e_i+\s_{ki}\overline {e_i})=\s_{k \bar i}H(e_i)+\s_{ki}H(\overline {e_i})\\
&= \s_{k \bar i}(\s_{i \bar j}e_j+\s_{ij}\overline {e_j})+\s_{ki}\overline{(\s_{i \bar j}e_j+\s_{ij}\overline {e_j})}\\
&=(\s_{k \bar i}\s_{i \bar j}+\s_{ki}\overline{\s_{ij}})e_j+(\s_{k \bar i}\s_{ij}+\s_{ki}\overline{\s_{i \bar j}})\overline {e_j}.
\end{aligned}\label{rc}
\end{equation}
\begin{equation}
\begin{aligned}
g(R_{\mathcal{N}}(e_k), \overline{e_l})&=R(e_k,\mathcal{N},\mathcal{N}, \overline{e_l})=\frac{1}{2}R(e_k,e_n+\overline{e_n},e_n+\overline {e_n}, \overline {e_l})\\
&= \frac{1}{2}R(e_k, \overline{e_n}, e_n, \overline{e_l})=\frac{1}{2}R_{n \bar n k \bar l}.\\
\end{aligned}\label{gj}
\end{equation}

After adding equations (\ref{wd}), (\ref{rc}), and (\ref{gj}), we get
\begin{equation}
g((\nabla_{\mathcal{N}} H +H^2 +R_{\mathcal{N}})e_k, \overline {e_l})=\frac{\partial}{\partial t} (\s_{k \bar l})+\s_{k \bar i}\s_{i \bar l}+\s_{ki}\overline{\s_{il}}+\frac{1}{2}R_{n \bar n k \bar l}=0.
 \label{bi}
\end{equation}
Let $k=l=n$. Note that $\s_{n n}=-\s_{n \overline n}$ since $\nabla_{e_n} e_n+\nabla_{\overline{e_n}} \,e_n=0$. Therefore
\begin{equation}
     (2\s_{n \overline n})'+(2\s_{n \overline n})^2+\widetilde{R}_{n \overline{n} n  \overline{n}} \leq 0.    \label{he1}
\end{equation}
It follows from \cite[Lemma 2.2]{TY2012} that
 \begin{equation}\label{he2}
     \lim_{\s \to 0+}(2\s_{n \bar n}-\frac{1}{\s})=0.
 \end{equation}
 Then we have the following Hessian estimates, see also Tam-Yu \cite[Theorem 2.1 on p.488]{TY2012} for related results.
 \begin{proposition}\label{Hessian1}
    Let $(\M^{n},J,\g)$ be a complete K\"ahler manifold with complex dimension $n$, and $M^{n-r} \subset \widetilde{M}^n$ a compact complex submanifold with the induced metric $g$. Let $\s$ be the distance function of any point $q \in \M$ away from $M$. Let $e_n=\frac{1}{\sqrt 2}(\nabla\s -\sqrt{-1}J(\nabla\s))$. If $R_{n \overline n n \overline n}\ge -k(\s)$ holds on $\M \setminus (M \cup  Cut(M))$ for some function $k(t)$ with $0 \leq k(t) \in C^0[0, \infty)$, then we have
     \begin{equation}
         \s_{n \overline {n}}\le \frac{u'(\s)}{2u(\s)},  \ \ x \in \M \setminus (M \cup  Cut(M)). \label{he3}
     \end{equation}
     Here $u(t)$ is determined by (\ref{res}), i.e. $u(t) \in C^2[0, \infty)$ is the solution to
    $u''(t)-k(t)u(t)=0$ with the initial conditions $u(0)=0$ and $u'(0)=1$.
\end{proposition}
 \begin{proof}[Proof of Proposition \ref{Hessian1}]
 It follows from (\ref{res}) that
 \[
 (\frac{u'}{u})'+(\frac{u'}{u})^2-k=0.
 \]
 We solve the Taylor expansion near $t=0$ as
 \begin{equation}
    \frac{u'(t)}{u(t)}-\frac{1}{t}=\frac{tu'(t)-u(t)}{tu(t)}=\frac{t(1+u^{\prime\prime}(\xi_1)t)-(t+\frac{1}{2}u^{\prime\prime}(\xi_2)t^2)}{t^2+o(t^2)} \rightarrow \frac{1}{2}u''(0)=0.  \label{initial}
\end{equation}
     It is clear that (\ref{he3}) follows from (\ref{he1}), (\ref{he2}), (\ref{initial}), Lemma \ref{eta}, and Lemma \ref{rct}.
 \end{proof}

\subsubsection{Proof of dimension upper estimates}

\begin{proposition}[based on {\cite[Theorem 8]{Liu2016}}]\label{thrcir}

     Let $(\M^{n},J,\g)$ be a complete noncompact K\"ahler manifold and $(M^{n-r},g)$ be a compact K\"ahler submanifold with the induced metric. Let $\s$ be the distance function from $M$.  Let $e_n=\frac{1}{\sqrt{2}}(\nabla \s-\sqrt{-1}J\nabla\s)$ defined in $\M \setminus (M \cup Cut(M))$. For any given $r>0$, we define
  \[
  M_f(r)=\sup\{ |f(x)| \ \ |\ \ x\in T(M, r)\},\ \text{where}\ T(M, r)=\{q \in M\ |\ \s(q)<r\}.
  \]
  Suppose that $R_{n \overline n n \overline n}(x)\ge -k(\s(x))$ holds on $\M \setminus (M \cup Cut(M))$, where $0 \leq k(t) \in C^0[0, \infty)$. Let $h(t) \in C^2(0,+\infty)$ satisfies $h'(t)=\frac{1}{u(t)}$ with $u$ defined in (\ref{res}). Then $\forall f \in \mathcal{O}(\M)$ and any $0<\s_1 <\s_2 <\s_3$,
    \begin{equation}
        \log M_f(\s_2)\le \frac{h(\s_3)-h(\s_2)}{h(\s_3)-h(\s_1)}\log M_f(\s_1) + \frac{h(\s_2)-h(\s_1)}{h(\s_3)-h(\s_1)}\log M_f(\s_3).  \label{thr1}
    \end{equation}
\end{proposition}

\begin{remark}
    (\ref{thr1}) is called a three-circle type inequality and was first proved by Liu \cite[Theorem 8 on p.2913]{Liu2016} as a generalization of \cite[Theorem 1]{Liu2016} which deals with K\"ahler manifolds with nonnegative holomorphic sectional curvature. Liu's beautiful results have found many applications in the study of complete K\"ahler manifolds with nonnegative bisectional curvature. (\ref{thr1}) was further generalized to almost Hermitian manifolds by Yu-Zhang in \cite[Theorem 1.3 on p.184]{YZ2022}, where applications to Liouville type theorems on almost Hermitian manifolds were also derived. The main difference of Proposition \ref{thrcir} with results in \cite{Liu2016} and \cite{YZ2022} is that we work with the distance to a closed submanifold. We also follow {\cite[Theorem 8]{Liu2016}} to add some detailed explanation on the Calabi trick which is used to deal with cut-focus points.
\end{remark}

\begin{proof}[Proof of Proposition \ref{thrcir}]

For fixed $\s_1 <\s_3$, we may assume $M_f(\s_3)>M_f(\s_1)>0$. Now we introduce
    \begin{align}
        F(x)&=[h(\s_3)-h(\s(x))]\log M_f(\s_1) + [h(\s(x))-h(\s_1)]\log M_f(\s_3), \label{eqf} \\
        G(x)&=(h(\s_3)-h(\s_1))\log |f(x)|.   \label{eqg}
    \end{align}
    Then $G(x) \le F(x)$ when $\s(x)=\s_1$ or $\s(x)=\s_3$. Our goal is to apply the maximum principle to some perturbation of $G-F$ inside the region $T(M,\s_3) \setminus \overline{T(M,\s_1)}$.

    Let $h_{\varepsilon}=\log(e^h-\varepsilon)$ and we define $F_{\varepsilon}$ and $G_{\varepsilon}$ by replacing $h$ in (\ref{eqf}) and (\ref{eqg}) with $h_{\varepsilon}$ respectively. Assume that $G_{\varepsilon}-F_{\varepsilon}$ gets a positive maximum at some point $x_{\varepsilon} \in T(M,\s_3) \setminus \overline{T(M,\s_1)}$. We may assume $f(x_{\varepsilon}) \neq 0$. It follows that $\sqrt{-1}\partial \bar \partial \log |f(x)|_{x=x_{\varepsilon}}=0$.

    \textbf{Case 1:} Suppose $x_{\varepsilon}$ is not a cut-focal point of $M$. We may compute
    \begin{equation}
        h(\s(x))_{n \overline n}=h''(\s)\s_n \s_{\overline n}+h'(\s)\s_{n \overline n}=-\frac{u'}{2u^2}+\frac{1}{u}\s_{n \overline n} \le 0.    \label{hessianh1}
    \end{equation}
    Note that in the last inequality we use $(\ref{he3})$. Similarly we solve
    \begin{equation}
        h_{\varepsilon}(\s(x))_{n \overline n}|_{x=x_{\varepsilon}}=(1+\frac{\varepsilon}{e^h-\varepsilon})h(\s(x))_{n \overline n}-\frac{\varepsilon e^h (h')^2}{2(e^h-\varepsilon)^2} <0.
        \label{hessianh2}
    \end{equation}
    It follows that
    \[
    (G_{\varepsilon}-F_{\varepsilon})_{n \bar n}|_{x=x_{\varepsilon}}= -h_{\varepsilon}(\s(x))_{n \bar n}|_{x=x_{\varepsilon}} (\log M_f(\s_3)-\log M_f(\s_1))>0.
    \]
    This is a contradiction as $G_{\varepsilon}-F_{\varepsilon}$ achieve a positive maximum at $x_{\varepsilon}$.

    \textbf{Case 2:} Suppose  $x_{\varepsilon}$ is a cut-focal point of $M$. There exists a (not unique) minimal geodesic $\gamma(t):[0,a] \to \M$ connecting $\gamma(a)=x_{\varepsilon}$ and a point $\gamma(0)=p \in M$. For any $\delta>0$, let $p_{\delta}=\gamma(\delta)$ and $\widehat{s}$ be the distance function from $p_{\delta}$. By the triangle inequality,
    \begin{align*}
        \s(x) \le \widehat{s}(x)+\delta, \ \ \s(x_{\varepsilon})=\widehat{s}(x_{\varepsilon})+\delta.
    \end{align*}
    We define a support function of $F_{\varepsilon}$ in the following
    \begin{equation*}
        F_{\varepsilon,\delta}(x)=[h_{\varepsilon}(\s_3)-h_{\varepsilon}(\widehat{s}(x)+\delta)]\log M_f(\s_1) + [h_{\varepsilon}(\widehat{s}(x)+\delta)-h_{\varepsilon}(\s_1)]\log M_f(\s_3).
    \end{equation*}
    It follows that $G_{\varepsilon}-F_{\varepsilon,\delta} \le G_{\varepsilon}-F_{\varepsilon}$ and $(G_{\varepsilon}-F_{\varepsilon,\delta})(x_{\varepsilon}) = (G_{\varepsilon}-F_{\varepsilon})(x_{\varepsilon})$. Then $G_{\varepsilon}-F_{\varepsilon,\delta}$ gets a local maximum at $x_{\varepsilon}$. As $x_{\varepsilon}$ is not a cut point of $p_{\delta}$, $f_n=\frac{1}{\sqrt{2}}(\nabla \widehat{s}-J \nabla \widehat{s})$ is well-defined in a neighborhood of $\gamma(t)$ with $t \in (\delta, a]$. We claim that for $\delta>0$ small enough,
    \begin{equation}
        (G_{\varepsilon}-F_{\varepsilon,\delta})_{f_n \overline{f_n}}(x_{\varepsilon})>0,   \label{hessianclaim}
    \end{equation}
    which leads to a contradiction again.

    From (\ref{he1}), we have
    \begin{equation}
        (2\widehat{s}_{n\bar{n}})^{\prime}+(2\widehat{s}_{n\bar{n}})^2+R_{n\bar{n}n\bar{n}}(\gamma(t+\delta)) \leq 0,     \label{hatsricatti}
    \end{equation}
   where
   \begin{equation}
       R_{n\bar{n}n\bar{n}}(\gamma(t+\delta)) \geq -k(\widetilde{s}(\gamma(t+\delta)))\ \text{for any}\ t \in (0, a-\delta].   \label{newlbound}
   \end{equation}
   Indeed, we note that the vector fields $f_n$ and $e_n$ coincide along $\gamma(t)$ with $t \in (\delta, a)$. So the curvature bound (\ref{newlbound}) still holds when $x_{\varepsilon}=\gamma(a)$ after we take limit $t \rightarrow (a-\delta)-$.
   We consider two initial value problems to ODEs on $[0, a-\delta]$
   \begin{align}
   &u^{\prime\prime}(t)-k(t)u(t)=0,\ \ u(0)=0,\ u'(0)=1,   \ \ \text{which is exactly}\ (\ref{res});\\
   &V^{\prime\prime}(t)-k(t+\delta)V(t)=0,\ \ V(0)=0,\ V'(0)=1.
   \end{align}
    By the continuous dependence property of ODE solutions with respect to coefficients, there exists some constant $\lambda(\delta)$ with $\lim_{\delta \rightarrow 0} \lambda(\delta)=0$ so that
     \begin{equation}
         \sup_{[0, a-\delta]} (|u(t)-V(t)|+|u'(t)-V'(t)|) \leq \lambda(\delta).   \label{ODEcompare}
     \end{equation}
     It follows from (\ref{ODEcompare}) that
     \begin{equation}
     \frac{u'(a-\delta)-\lambda(\delta)}{u(a-\delta)+\lambda(\delta)}  \leq  \frac{V'(a-\delta)}{V(a-\delta)} \leq \frac{u'(a-\delta)+\lambda(\delta)}{u(a-\delta)-\lambda(\delta)}.
     \end{equation}
     Note that $\frac{V'(t)}{V(t)}$ satisfies the following Ricatti type equation
    \begin{equation}
        (\frac{V'(t)}{V(t)})^{\prime}+(\frac{V'(t)}{V(t)})^2-k(t+\delta)=0.  \label{Vricatti}
    \end{equation}
    We may apply Lemma \ref{rct} on (\ref{hatsricatti}) and (\ref{Vricatti}) to conclude that
    \begin{equation}
        \widehat{s}_{f_n\overline{f_n}}|_{x_{\varepsilon}} \leq \frac{V'(a-\delta)}{2V(a-\delta)} \leq \frac{u'(a-\delta)+\lambda(\delta)}{2(u(a-\delta)-\lambda(\delta))}.
    \end{equation}
    Compared with (\ref{hessianh1}), we arrive at
    \begin{equation}
        h(\widehat{s}+\delta)_{f_n \overline{f_n}}|_{x=x_{\varepsilon}}=-\frac{u'}{2u^2}+\frac{1}{u} \widehat{s}_{f_n \overline{f_n}}
        \leq  -\frac{u'(a)}{2(u(a))^2}+\frac{u'(a-\delta)+\lambda(\delta)}{2u(a)[u(a-\delta)-\lambda(\delta)]}.
    \end{equation}
    It follows from (\ref{hessianh2}) that for any fixed $\epsilon>0$, we may pick $\delta>0$ small enough so that
    \begin{equation}
        h_{\varepsilon}(\widehat{s}(x)+\delta)_{f_n \overline{f_n}}|_{x=x_{\varepsilon}}=(1+\frac{\varepsilon}{e^h-\varepsilon})h(\widehat{s}+\delta)_{f_n \overline{f_n}}|_{x=x_{\varepsilon}}-\frac{\varepsilon e^h (h')^2}{2(e^h-\varepsilon)^2} <0.    \label{hessianh3}
    \end{equation} So the claim (\ref{hessianclaim}) is proved.

    To sum up, $G_{\varepsilon}(x) \leq F_{\varepsilon}(x)$ holds, so does (\ref{thr1}) after we take $\varepsilon \to 0$.

\end{proof}

\begin{proof}[Proof of Proposition \ref{upperdim}]
Let $X$ be the total space of a line bundle $\pi: L \rightarrow M$ over a compact K\"ahler manifold $M$. Assume $X$ admits a complete K\"ahler metric $\widetilde{g}$ which satisfies \ref{AssumeA}. In below the constants $C$ may differ from line to line.

For any $f \in \mathcal{O}_d(X)$, we have $M_f(\s)\le C\s^d$ for some constant $C>0$. If $\int_0^{+\infty} sk(s)\,ds < \infty$, it follows from Lemma \ref{eta} that
\begin{align}
\frac{1}{\eta}\log \s+C_1 \le h(\s) \le \log\s+C_2.    \label{hbound}
\end{align}
After taking $\s_3 \to +\infty$ in $(\ref{thr1})$, we arrive at
\begin{align}
   \log M_f(\s_2) &\le \log M_f(\s_1) + (h(\s_2)-h(\s_1))\limsup_{\s \to +\infty} \frac{\log M_f(\s)}{h(\s)}  \nonumber\\
   &\le \log M_f(\s_1) + (h(\s_2)-h(\s_1))\eta d.   \label{thr2}
\end{align}
It follows that $\dfrac{M_f(\s)}{e^{\eta d h(\s)}}$ is non-increasing with respect to $\s \in (0, \infty)$. Argue similarly as in (\ref{initial}), we have
\[
h'(t)-\frac{1}{t}=\frac{t-u(t)}{tu(t)}=\frac{-\frac{u^{\prime\prime}(\xi_1)}{2}+O(t)}{1+\frac{u^{\prime\prime}(\xi_2)}{2}t+O(t^2)}=o(1).
\]
After integration, we have
\begin{equation}
    h(\s)=\log\s+C+o(\s),
\end{equation}
and
\begin{equation}
    \lim_{\s \to 0} \frac{e^{\eta d h(\s)}}{\s^{\eta d}}=C>0.
\end{equation}
It follows from the monotonicity of $\dfrac{M_f(\s)}{e^{\eta d h(\s)}}$ that
    \begin{equation}
        \lim_{\s \to 0} \frac{M_f(\s)}{\s^{\eta d}}=C>0.   \label{vanish1}
    \end{equation}
The point is to show that (\ref{vanish1}) implies that the vanishing order of $f$ along $M$ is bounded from above by $\lfloor \eta d \rfloor$.

To that end, let $\varphi_i: \pi^{-1}(U_i) \rightarrow U_i \times \mathbb{C}$ be the local trivialization of  $\pi: L \rightarrow M$, and $v=p_2 \circ \varphi_i$ where $p_2$ is the projection to the second factor. Any holomorphic function on $X$ admits a Taylor expansion as in (\ref{taylor}). By the vanishing order of $f$ along $M$ we mean the lowest order of $v$ term in the Taylor expansion of $f$. Now we observe there exists some $\s_0>0$ so that
\begin{align}
|v|<C \s\  \ \text{for}\ \ \s < \s_0.    \label{vanish2}
\end{align}
Let $(x_1, x_2, \cdots, x_{2n-1}, x_{2n})$ denote the Fermi coordinates on $X$. Then we have $\s^2=x_{2n-1}^2+x_{2n}^2$. We refer to \cite[Chapter 2]{Gray} for a detailed discussion on Fermi coordinates.

If we write $v=v_1+\sqrt{-1}v_2$, then
\begin{align*}
    v_1&=\alpha_1(x_1, \cdots, x_{2n-1}) x_{2n-1}+ \beta_1(x_1, \cdots, x_{2n-1})x_{2n}+O(x_{2n-1}^2+x_{2n}^2),\\
    v_2&=\alpha_2(x_1, \cdots, x_{2n-1}) x_{2n-1}+ \beta_2(x_1, \cdots, x_{2n-1})x_{2n}+O(x_{2n-1}^2+x_{2n}^2).
\end{align*}
Then (\ref{vanish2}) follows and
\begin{equation}
        \liminf_{\s \to 0} \frac{M_f(\s)}{|v|^{\eta d}}>0.   \label{vanish3}
\end{equation}
In view of the Taylor expansion (\ref{taylor}), the desired dimension estimates (\ref{dimcmain1}) follows.

Suppose
\begin{equation}
\operatorname{dim} \mathcal{O}_d(X, \widetilde{g}) = \sum_{k=0}^{\lfloor \eta d \rfloor} \operatorname{dim} H^0(M, L^{-k}).   \label{vanish4}
\end{equation}
holds along a infinite sequence $\{d_j\}_{j=1}^{\infty} \subset \mathbb{R}^{+}$. By Lemma \ref{eta}, we have $1 \leq u'(t) \leq \eta=e^{\int_0^{\infty} tk(t)dt}$ and $\lim_{t \rightarrow \infty} u'(t) \leq \eta$. We claim that (\ref{vanish4}) forces the latter inequality becomes an equality. Hence $k(t)$ is identically zero on $[0, \infty)$ by Lemma \ref{eta}. Otherwise, there exists some $\delta>0$ small enough so that $\lim_{t \rightarrow \infty} u'(t) \leq \eta-\delta$. As a result, (\ref{hbound}) still holds after we replace $\eta$ by $\eta-\delta$. Following the previous argument which leads to (\ref{vanish3}), we get that the vanishing order of any $f \in \mathcal{O}_d(X)$ can be bounded by $\lfloor (\eta-\delta) d \rfloor$. Compared with (\ref{vanish4}), we get
\begin{equation}
\sum_{k=0}^{\lfloor \eta d_j \rfloor} \operatorname{dim} H^0(M, L^{-k})= \operatorname{dim} \mathcal{O}_{d_j}(X, \widetilde{g}) \leq \sum_{k=0}^{\lfloor (\eta-\delta) d_j \rfloor} \operatorname{dim} H^0(M, L^{-k}) .   \label{vanish5}
\end{equation}
It is a contradiction after we take $d_j \rightarrow \infty$.
\end{proof}

\begin{remark}
    In Lemma \ref{eta}, we may relax $\int_0^{\infty} sk(s)\,ds <\infty$ by $\int_0^t sk(s)\,ds \le K(t)$, then we get $u'(t) \le e^{K(t)}$. In Proposition \ref{thrcir} we choose $h(t)$ correspondingly as
    \begin{equation}
         h'(t) \ge \frac{1}{\exp(\int_0^te^{K(s)}\,ds)}.   \label{hdefgen}
    \end{equation}
    For instance, if $k(s)=\frac{\varepsilon}{s^2\log s}$ for $s>0$ large and $0<\varepsilon<1$, then we have $h(t)\ge C_1 (\log t)^{1-\varepsilon}+C_2$.
    The proofs of Propositions \ref{thrcir} and \ref{upperdim} imply the following Liouville property.

    Assume that $(X, \widetilde{g})$ satisfies the assumption of Proposition \ref{upperdim} except that $k(s)$ is replaced by $k(s)=\frac{\varepsilon}{s^2\log s}$ with $0<\varepsilon<1$ and $s$ large enough. Then any $f \in \mathcal{O}(X)$ satisfying
    \begin{align}
     \lim_{\s \to +\infty}\frac{\log M_f(\s)}{(\log \s)^{1-\varepsilon}}=0   \label{Lgrowth1}
    \end{align}
    must be constant. The statement follows after we take $\s_3 \to +\infty$ in $(\ref{thr1})$. This type of Liouville theorem for holomorphic (and plurisubharmonic) function was considered by Takegoshi (\cite{Take1990} and \cite{Take1993}) when $X$ is a complete K\"ahler manifold with a pole.

    It is interesting to note the role of the value of $\varepsilon$ in $k(s)=\frac{\varepsilon}{s^2\log s}$. If $\varepsilon=1$, i.e. $k(s)=\frac{1}{s^2\log s}$ for $s>0$ large, then $h(t)\ge \log\log t+C$ by (\ref{hdefgen}). Then the above statement holds if we replace (\ref{Lgrowth1}) by
    \begin{align}
    \lim_{\s \to +\infty}\frac{\log M_f(\s)}{\log\log \s}=0.  \label{Lgrowth2}
    \end{align}
    However, if $\varepsilon>1$ i.e. $k(s)=\frac{\varepsilon}{s^2\log s}$ for $s>0$ large, the lower bound of $h$ (defined by (\ref{hdefgen})) cannot tend to $\infty$ as $s \rightarrow \infty$. We fail to get any Liouville property by Proposition \ref{thrcir}. Such a dichotomy matches Milnor's result \cite{Milnor} on rotationally symmetric simply connected noncompact Riemann surfaces.
\end{remark}

\subsection{Dimension estimates: lower bounds}

We begin with a lemma from the beautiful work of Li-Tam \cite{LT1991}. Let $g=u(dx^2+dy^2)$ be a complete K\"ahler metric on a complex plane $\mathbb{C}=\{\,z=x+y\sqrt{-1}\,\}$. Assume that its Gauss curvature $K=-\dfrac{\Delta \log u}{2u}\le 0$, i.e. $\log u$ is subharmonic. Let $B_g(O,r)$ be the geodesic ball with radius $r$ centered at the origin and $B(O,r)$ be the Euclidean ball with the same radius.

\begin{lemma}[Li-Tam {\cite[p.147]{LT1991}}]\label{sup-inf}
Define $s(r)=\sup_{z\in \partial B_g(O,r)}|z|$ and $i(r)=\inf_{z\in \partial B_g(O,r)}|z|$. There exists some $r_0>0$ and some $\delta \in C^0([r_0, \infty))$ with $0 \leq \delta(r) \leq 2\pi$ and $\lim_{r \rightarrow \infty} \delta(r)=0$ such that $B(O, e)\subset B_g(O, r_0)$. Moreover, for any $R>r>r_0$, we have
\begin{equation}
        \log s(r) \le \exp\Big(\frac{\sqrt{\delta(r)}} {2\sqrt{\int_r^R \frac{1}{L(t)}dt}}\Big)  \log i(R).
        \label{Cbound0}
\end{equation}
Here $L(t)$ denotes the length of $\partial B_g(O, r)$.
\end{lemma}

\begin{proof}[Proof of Lemma \ref{sup-inf}]
The estimate (\ref{Cbound0}) is derived in the course of the proof of the main result {\cite[Theorem 2.5]{LT1991}}. For the sake of convenience we include a proof.

Given any $C^1$ function $f$, we apply the Cauchy-Schwarz inequality
\[
\Big(\int_{\partial B_g(O, t)}|\nabla f|dL\Big)^2\le L(t)\Big(\int_{\partial B_g(O, t)}|\nabla f|^2dL\Big).
\]
Then we integrate both sides from $r$ to $R$ after divided by $L(t)$
\begin{equation}
\int_r^R \Big(\int_{\partial B_g(O, t)}|\nabla f|dL\Big)^2\frac{1}{L(t)}\,dt\le \int_r^R \Big(\int_{\partial B_g(O, t)}|\nabla f|^2dL\Big)\,dt\le \int_{\C \setminus B_g(O, r)} |\nabla f|^2 dV_g.
\label{Cbound1}
\end{equation}

There exist $x,y \in \partial B_g(O, t)$ such that $f(x)=\sup_{\partial B_g(O, t)} f$ and
$f(y)=\inf_{\partial B_g(O, t)} f$. $x$ and $y$ divide $\partial B_g(O, t)$ into two curves and we have
\begin{equation}
2(f(x)-f(y))\le \int_{\partial B_g(O, t)}|\nabla f|dL.  \label{Cbound2}
\end{equation}

Next we choose $f=\log \log |z|$. Note that $f \in C^{\infty} (\mathbb{C} \setminus B(O, e))$ and
\[
\int_{\C \setminus B(O,e)} |\nabla f|^2 dxdy=2\pi \int_e^{\infty} \frac{1}{t\log^2 t}\,dt=2\pi.
\]
Notice that $\int_{\mathbb{C}}|\nabla f|^2 dV_g$ is a conformal invariant of $g$. There exists some $r_0>0$ and some $\delta \in C^0([r_0, \infty))$ with $0 \leq \delta(r) \leq 2\pi$ and $\lim_{r \rightarrow \infty} \delta(r)=0$ with the property that
\begin{equation}
    \int_{\C \setminus B_g(O,r)} |\nabla f|^2 dV_g \le \delta(r).  \label{deltadef}
\end{equation}
It follows from (\ref{Cbound1}) and (\ref{Cbound2}) that
\begin{equation}
4\int_r^R (\log\log s(t)-\log\log i(t))^2\frac{dt}{L(t)} \le \delta(r).  \label{Cbound3}
\end{equation}
Given any $0<t_1<t_2$. As $\ln|z|$ is harmonic on away from the origin, we may apply the maximum principle on $B_g(O, t_2) \setminus B_g(O, t_0)$ for $t_0$ sufficiently small. Then we get $s(t_1) \leq s(t_2)$, which implies that $s(t)$ is increasing on $t$. Similarly, we may consider $B_g(O, t_3) \setminus B_g(O, t_1)$ for some $t_3$ large enough to get $i(t_1) \leq i(t_2)$, i.e. $i(t)$ is also increasing on $t$. Then we get
\begin{equation}
4(\log\log s(r)-\log\log i(R))^2\int_r^R \frac{dt}{L(t)} \le \delta(r).  \label{Cbound4}
\end{equation}
The desired estimate (\ref{Cbound0}) follows from (\ref{Cbound4}).
\end{proof}

\begin{proof}[Proof of Proposition \ref{lowerdim}]

Recall $(X, \widetilde{g})$ satisfies \ref{AssumeA} and \ref{AssumeB}. And the holomorphic sectional curvature satisfies $-k(\widetilde{s}) \leq R(e_n, \overline{e_n}, e_n, \overline{e_n}) \leq 0$ on $(0, \infty)$. Moreover, for each $p \in M$, $\pi^{-1}(p)$ is totally geodesic in $(X, \widetilde{g})$.

Under these assumptions, the Gauss curvature of the induced metric $g_p=\widetilde{g}|_{\pi^{-1}(p)}$ on each $\pi^{-1}(p)$ satisfies $-k(s) \leq K(g) \leq 0$. From now on, we work on $(\pi^{-1}(p), g_p)$. We also write $g$ for short, assuming the notation is clear from context. Let $p$ denote the origin of $\pi^{-1}(p)$, $A(t)$ the area of $B_g(p, t)$, $L(t)$ the length of $\partial B_g(p, t)$
\begin{align}
s(t)=\sup_{v\in \partial B_{g}(p,t)}|v|,\ \ \ \ i(t)=\inf_{v\in \partial B_g(p,t)}|v|.   \label{supinfdef}
\end{align} where $v$ denote holomorphic coordinates in the fiber direction for any point in $\pi^{-1}(p)$. Let $B(p, t)$ stand for the disk $\{|v|<t\} \subset \pi^{-1}(p)$. Under those notations, we may also write $g_p=u |dv|^2$ for some smooth function $u$. Using standard results on comparison geometry (Lemma \ref{eta} and \ref{rct}), we have
\begin{equation}
    2\pi t  \leq  L(t) \leq  \eta 2\pi t,\ \ \pi t^2  \leq  A(t) \leq  \eta \pi t^2.   \label{ACbound1}
\end{equation}Here $\eta$ is defined in (\ref{etadef}).
Motivated by {\cite[p.147]{LT1991}}, we find some $R>r$ so that
\[
1=\int_r^R \frac{dt}{L(t)}\ge \frac{1}{2\eta\pi}\int_r^R \frac{dt}{t}=\frac{1}{2\eta\pi}(\log R-\log r).
\]
Therefore
\begin{equation}
    R\le e^{2 \eta\pi} r.   \label{ACbound2}
\end{equation}
For such a particular choice of $R$, (\ref{Cbound0}) is reduced to
\begin{equation}
    s(r)\le i(R)^{\exp\Big(\frac{1}{2}\sqrt{\delta(r)}\Big)}.  \label{ACbound3}
\end{equation}
By definition of $i(r)$, we have $B(p,i(r)) \subset B_g(p, r)$. Then
\[
\operatorname{Vol}_g(B(p, i(r))) \le \operatorname{Vol}_g(B_g(p, r)).
\]
In the meantime, we estimate
\[
\operatorname{Vol}_g(B(p, i(r)))=\int_{B(p, i(r))}u\,dxdy\ge u(p)\pi(i(r))^2.
\]
Here we make a crucial use of the fact that $\log u$ (hence $u$) is subharmonic.
To sum up, we arrive at
\begin{equation}
    u(p)(i(r))^2 \le \frac{A(r)}{\pi} \leq \eta  r^2.  \label{ACbound4}
\end{equation}
From (\ref{ACbound2}), (\ref{ACbound3}), and (\ref{ACbound4}), we obtain
\begin{align}
\log s(r) &\leq e^{\frac{\sqrt{\delta(r)}}{2}} \log i(R)  \nonumber\\
      & \leq e^{\frac{\sqrt{\delta(r)}}{2}}  \log  \Big( \sqrt{\frac{\eta}{u(p)}} R\Big)  \nonumber\\
      &  \leq   e^{\frac{\sqrt{\delta(r)}}{2}}  \log r + e^{\frac{\sqrt{\delta(r)}}{2}}  (2\eta \pi +\frac{1}{2}\log \frac{\eta}{u(p)}).  \label{ACbound5}
\end{align}

Recall the definition of $\delta(r)$ in (\ref{deltadef}). Given $f=\log (\log |v|)$, we could write $\int_{\C \setminus B_g(p,r)}  |\nabla f|^2 dV_g$ as $\int_{\C \setminus B_{g}(p,r)}  |\nabla f|^2 dV_{g_e}$ where $g_e=|dv|^2$ is the Euclidean metric on $\pi^{-1}(p)$ using conformal invariance. For any compact set $K \subset U$ where $U$ is a chart from the local trivialization of $L$ as $\varphi_i: \pi^{-1}(U_i) \rightarrow U_i \times \mathbb{C}$. We conclude that $B_{g}(0,r)$ varies uniformly continuous with respect to $p \in K$. Therefore $\delta(r)$ converges uniformly to zero as $r \rightarrow \infty$.

To sum up, for any $\varepsilon>0$, there exists two constants $C$ and $R_0$ which depend on $\varepsilon>0$, $\eta$, $u(p)$, and the geometry of $g_p$ such that for any $p \in K$,
    \begin{equation}
        \sup_{B_g(p, r)} |v| \leq C r^{1+\varepsilon}.\ \ \ r>R_0.  \label{ACbound7}
    \end{equation}

From the Taylor expansion (\ref{taylor}), we deduce that for any $\varepsilon>0$,
\begin{equation}
\operatorname{dim} \mathcal{O}_d(X, \widetilde{g}) \geq \sum_{k=0}^{\lfloor (1+\varepsilon)d \rfloor} \operatorname{dim} H^0(M, L^{-k}).   \label{ACbound8}
\end{equation}
(\ref{dimcmain2}) is proved after we take $\varepsilon \rightarrow 0$ in (\ref{ACbound8}).

Finally, we exmaine the equality case of (\ref{dimcmain2}) along a sequence $d_j \rightarrow \infty$. We focus on any fixed fiber $(\pi^{-1}, g_p)$. Note that it is a copy of complex plane with nonpositive Gauss curvature and finite total curvature. The latter property follows from (\ref{ACbound1})
\[
\int_{\pi^{-1}(p)} K(g_p) \operatorname{dVol}_{g_p}  =\int_0^{\infty}  k(s)L(s)ds \leq 2\pi \eta \int_0^{\infty}  sk(s)ds.
\]
For such surfaces, a remarkable result \cite[Corollary 3.3]{LT1991} states that $\lim_{\s \rightarrow \infty} \frac{\ln |\s|}{\ln |v|}$ exists and
\begin{equation}
\lim_{\s \rightarrow \infty} \frac{\ln |\s|}{\ln |v|}=1-\frac{1}{2\pi}  \int_{\pi^{-1}(p)} K(g_p) \operatorname{dVol}_{g_p}.    \label{LTmain}
\end{equation}
In view of (\ref{LTmain}), it suffices to show show $\lim_{\s \rightarrow \infty} \frac{\ln |v|}{\ln |\s|}=1$. It remains to show $\limsup_{s \rightarrow \infty} \frac{\ln |v|}{\ln s} \geq 1$ after we take (\ref{ACbound7}) into account. Assume the contrary, and then there exists some $\varepsilon_0$ and $\s_0$ so that
$\ln |v| \leq (1-\varepsilon_0) \ln|\s|$ for $\s \geq \s_0$. Then for any $d_j$ large enough, there exists $C>0$ and $\s_1$ so that
\[
|v|^{d_j+1} \leq |v|^{\frac{d_j}{1-\varepsilon_0}} \leq C(|\s|^{d_j}+1),\ \ \s \geq \s_1.
\]
This is a contradiction as (\ref{dimcmain2}) is assumed to hold as an inequality along $\{d_j\} \rightarrow \infty$.

\end{proof}

\section{Liouville theorems on holomorphic mappings}\label{sec5}

\begin{lemma}\label{ctoc}
    Let $g_1$ and $g_2$ be two K\"ahler metrics on the complex plane $\mathbb{C}$, and $K_1$ and $K_2$ their Gauss curvatures respectively. Assume that $g_1$ is complete. Let $O$ denote the origin of $\mathbb{C}$. Assume that there exists some nonnegative function $k \in C^0([0, +\infty))$ with the property that $\int_0^{\infty} k(s)ds<\infty$ so that the Gauss curvatures of $g_1$ satisfies:
    \begin{align}
    K_1(x) \geq -k(s_1(x)), \ \text{where}\ s_1(x)=d_{g_1}(O, x),\ \text{and for any}\ x \in \mathbb{C}.
    \end{align}
    The Gauss curvature of $g_2$ satisfies $K_2 \leq 0$ everywhere on $\mathbb{C}$. For any entire holomorphic function $f  \in \mathcal{O}(\mathbb{C})$, we define:
    \begin{align}
    &M_f(s)=\sup_{x \in B_{g_1}(O, s)} d_{g_2}(O, f(x)),    \\
    &\mathcal{O}_d(\C,g_1,g_2)=\{ f\in \mathcal{O}(\C)\ |\ M_f(s)\le C s^d \ \text{for some constant}\   C>0\}.
    \end{align}
    If there exists some nonconstant $f \in \mathcal{O}_d(\C,g_1,g_2)$, then we have $d \geq \frac{1}{\eta}$ with $\eta$ defined in (\ref{etadef}). Moreover,
    \begin{equation}
        \lim_{s \to 0} \frac{M_f(s)}{s^{\eta d}}>0.    \label{dime2}
    \end{equation}
\end{lemma}
\begin{proof}[Proof of Lemma \ref {ctoc}]
    We only need to prove equation $(\ref{thr1})$. Let
    \begin{equation}
        G(x)=(h(\s_3)-h(\s_1))\log \phi(f(x)),   \label{adjust}
    \end{equation} as equation $(\ref{eqg})$ where we set $\phi(y)=d_{g_2}(O, y)$.
    For any $x \in \mathbb{C}$, we use $z$ the holomorphic coordinate, and $w$ the coordinate for $f(x)$.
    It does not hurt to assume that $\frac{\partial}{\partial z}$ is unitary at $x$ and $\frac{\partial}{\partial w}$ is unitary at $f(x)$. By the standard Hessian comparison theorem, we get $\phi_{w\overline{w}} \geq \frac{1}{2\phi}$ at $f(x)$. Therefore
    \begin{equation}
        (\log \phi(f(x)))_{z \overline z}=(\frac{\phi_{w \overline w}}{\phi}-\frac{\phi_w \phi_{\overline w}}{\phi^2})f_z  \overline{f_{z}} \geq (\frac{\phi_{1 \overline{1}}}{\phi}-\frac{1}{2\phi^2}) |f_z|^2\geq 0.
    \end{equation}
    In the meantime, we have (\ref{hessianh1}). Following the proof of Proposition \ref{thrcir}, we get
    $G(x) \le F(x)$. Therefore $\dfrac{M_f(s)}{e^{\eta d h(s)}}$ is non-increasing on $(0, \infty)$ as in (\ref{thr2}), and $(\ref{dime2})$ follows in view of the proof of (\ref{vanish1}).

\end{proof}

\begin{proposition}\label{xtoc}
    Assume $(X, \widetilde{g})$ is a complete K\"ahler manifold which satisfies the same assumption as in Proposition \ref{upperdim}. Assume $g^{\ast}$ is a K\"ahler metric with nonpositive Gauss curvature on $\mathbb{C}$. For any given $f \in \mathcal{O}(X)$, we define
    \[
   M_f(\s)=\sup_{x\in T(M,\s)}{d_{g^{\ast}}(0,f(x))},\ \  \text{where}\ T(M, r)=\{q \in M\ |\ \s(q)<r\}.
   \]
   Then the same conclusion as that of Lemma \ref {ctoc} holds.
\end{proposition}

\begin{proof}[Proof of Proposition \ref{xtoc}]
    It follows from the proof of Proposition \ref{thrcir} and Proposition \ref{upperdim}. We only need to make adjustments as in (\ref{adjust}).
\end{proof}

We prove Proposition \ref{xtoy} in the end of this section. As mentioned in the introduction, this type of result has potential implications for understanding holomorphic biholomorphism of $X$ in Proposition \ref{lowerdim}.

\begin{proof}[Proof of Proposition \ref{xtoy}]
    Note that there is no nonconstant entire curve on $M_2$ as it is Kobayashi hyperbolic. For any $f \in \mathcal{O}(X, Y)$ and $p\in X$, $\pi_2 \circ f(\pi_1^{-1}(p))$ must be a single point on $M_2$. Therefore we know $f$ maps fibers of X into fibers of $Y$. We may assume $f(\pi_1^{-1}(p)) \subset \pi_2^{-1}(q)$ and consider the restriction $f: \pi_1^{-1}(p) \rightarrow \pi_2^{-1}(q)$ with the induced metric from $(X,\widetilde{g})$ and $(Y,\widetilde{h})$ respectively. By Lemma \ref {ctoc}, we know any such restriction map is constant. Hence $f$ factors through $\pi_1$. Apriori the image of $f$ might wind in a complicated way and intersect some fibers of $M_2$ multiple times, and it remains unknown whether $f$ can be further reduced. Exactly, it is unclear whether there exists some $f_{\diamond} \in \mathcal{O}(M_1, M_2)$ and a global section $\sigma$ of ${L_2}|_{f_{\diamond}(M_1)}$ so that $f=\sigma \circ f_{\diamond} \circ \pi_1$. In the special case that the image of $\pi_2 \circ f$ is a single point, we know $f$ is constant by Proposition \ref{xtoc}.
\end{proof}

\appendix

\section{\texorpdfstring{$U(n)$}{TEXT}-invariant K\"{a}hler metrics on \texorpdfstring{$\mathbb{C}^{n}$}{TEXT} with negative curvature}  \label{negUN}

Wu-Zheng \cite{WZ} made an extensive study on $U(n)$-invariant K\"{a}hler metrics on $\mathbb{C}^{n}$ with $BI>0$. In particular, they demonstrated such examples with unbounded curvature and some connection between curvature decay and volume growth of these metrics. These examples serve as a testing ground for a broader investigation on complete K\"{a}hler manifolds with $BI \geq 0$. Motivated by \cite{WZ}, we study $U(n)$-invariant K\"{a}hler metrics on $\mathbb{C}^{n}$ with $BI<0$ and prove some asymptotic geometric properties of these metrics. In particular, we explain how we formulate the $\widetilde{\chi}$ function in Theorem \ref{main1}, and present some negatively curved K\"ahler metrics previously discovered by Milnor, Cao, and Seshadri in a uniform way.


Let $z=(z_1,\cdots,z_n)$ be the holomorphic coordinate on $\mathbb{C}^{n}$ and $r=|z|^{2}$.
A $U(n)$-invariant K\"{a}hler metric on $\mathbb{C}^{n}$ has the K\"{a}hler
form
\begin{equation}
\omega=\sqrt{-1}\partial \overline{\partial}
\mathcal{P}(r) \label{Kahler form}
\end{equation} for some $\mathcal{P} \in C^{\infty}
[0,+\infty)$. Under the local coordinates $\{z_i\}$, the metric has
components:
\begin{equation}
g_{i\overline{j}}=f(r)\delta_{ij}+f'(r) \overline{z}_i z_j.
\end{equation}
We further denote:
\begin{equation}
f(r)=\mathcal{P}'(r),\ \ \ \ h(r)=(rf)'.  \label{def of f and h}
\end{equation}
It can be checked that the from $\omega$ will give a complete K\"{a}hler metric on $\mathbb{C}^n$ if and
only if
\begin{equation}
f>0, \ \  h>0, \ \ \int_{0}^{+\infty}  \frac{\sqrt{h}}{\sqrt{r}}
dr=+\infty. \label{being complete}
\end{equation}

We compute the curvature tensor at
$(z_1,0,\cdots,0)$ under the unitary frame $\{
e_1=\frac{1}{\sqrt{h}} \partial_{z_1},e_2=\frac{1}{\sqrt{f}}
\partial_{z_2}, \cdots, e_n=\frac{1}{\sqrt{f}} \partial_{z_n} \}$.
We introduce three bisectional curvatures $A, B, C$ respectively:
\begin{equation}
A=R_{1\overline{1}1\overline{1}}=-\frac{1}{h} (\frac{rh'}{h})',\
B=R_{1\overline{1}i\overline{i}}=\frac{f'}{f^2}-\frac{h'}{hf},\
C=R_{i\overline{i}i\overline{i}}=2R_{i\overline{i}j\overline{j}}=-\frac{2f'}{f^2},
\label{ABC one form}
\end{equation} where we assume $2 \leq i \neq j \leq n$.
One may check all other components of curvature tensors are
zero except those which are equal to $A, B$, or $C$ due to symmetry properties of K\"ahler curvature tensors. Let $\mathcal{N}_n$ denote the set of complete
$U(n)$-invariant K\"ahler metrics on ${\mathbb C}^n$ with $BI<0$.
Our major goal in this appendix is to study geometric properties of K\"ahler metrics in $\mathcal{N}_n$ in the spirit of \cite{WZ}.

\subsection{Various characterizations of \texorpdfstring{${\mathcal N}_n$}{TEXT}}

In this subsection, we introduce various equivalent characterization of ${\mathcal N}_n$, all of which are related by a suitable change of variables. By choosing a suitable characterization we may study volume growth, curvature decay, and other geometric properties of K\"ahler metrics in $\mathcal {N}_n$ in a more convenient way.

\begin{theorem}
[\textbf{the $ABC$ function}] \label{mainApp}
Suppose $n \geq 2$ and $f$ is a smooth positive function on
$[0,+\infty)$ satisfying (\ref{being complete}). Then (\ref{Kahler
form}) gives a metric in ${\mathcal N}_n$ if and only if $A,B,C$ are negative.
Moreover, it has negative sectional curvature (or negative complex curvature operator) if in addition
\[
D \doteq AC-B^2>0\ \ \  (\,\text{or}\ D_{n} \doteq \frac{n}{2(n-1)} AC-B^2)
\]
holds on $\mathbb{C}^n$.
\end{theorem}

We define a smooth function $\xi$ on $[0,+\infty)$ by
\begin{equation}
\xi(r)=-\frac{r h'(r)}{h}. \label{def of xi}
\end{equation} On the other hand,
$\xi$ determines $h$ by $h(r)=h(0) e^{\int_{0}^{r}
-\frac{\xi(t)}{t} dt}$, hence $\omega$ up to scaling. It is exactly $\xi$ which motivated
our introduction of $\chi$ in (\ref{chidef}).

\begin{proposition}[\textbf{the $\xi$
function}]\label{propA2}
Let $n\geq 2$ and h be a smooth negative function on
$[0,+\infty)$. Then the form defined by (\ref{Kahler form}) gives a
complete K\'{a}hler metric with $BI<0$ ($BI \leq 0$) if and only if $\xi$ defined by (\ref{def of xi})
satisfying
\begin{align} \xi(0)=0,\ \ \ \xi^{\prime}<0 \ \ (\xi^{\prime}\leq
0).
\end{align}
\end{proposition}


Following \cite{WZ}, we introduce another function
$F$ in the following way. First we define $x=\sqrt{rh}$ and a
nonnegative function $y$ of $r$ by
\begin{equation}
y(0)=0, \ \ \ \ \ (x^{\prime}(r))^{2}-(y^{\prime}(r))^{2}=\frac{h}{4r}, \
\ \ \ y' > 0.
\end{equation}
Since $x(r)$ is strictly increasing, we may
define $F(x)$ by $y=F(x)$ with $x \in [0,+\infty)$. Extend $F$
to $(-\infty,+\infty)$ by letting $F(x)=F(-x)$. Starting with such a
$F$, one can recover the metric $\omega$. In fact it is more
convenient to use $p(x)=\sqrt{1-(F^{\prime}(x))^{2}}=\frac{1}{1-\xi}$.

\begin{proposition} [\textbf{the $p$ function}] \label{pfuncthm}
Suppose $n\geq 1$. Then any smooth function $p(x)$ defined on $[0,+\infty)$ satisfying
\begin{align}
p(0)=1,\ p^{\prime}(0)=0, \ p^{\prime\prime}(0)<0,
\ \ p(s)>0,\ \ p^{\prime}(s)<0,\ \forall s>0, \ \text{and}\ \int_{1}^{+\infty} \frac{p(\tau)}{\tau} d\tau=\infty,
\end{align}
corresponds to a metric in $\mathcal{N}_n$.
\end{proposition}

Let $v=rf$. We may rewrite the distance function $s$ in (\ref{being complete}) and $\operatorname{Vol}(B(s))$ in terms of $p(x)$
\begin{equation}
s=\int_{0}^{x} p(\tau) d \tau, \ \  v \doteq \int_{0}^{x} 2\tau
p(\tau) d\tau. \label{vbyp}
\end{equation}

\begin{align}
\operatorname{Vol}(B(s))=c_n v^n=c_n (\int_{0}^{x}  2\tau
p(\tau)d\tau)^n.  \label{volformula}
\end{align}
Here $c_n$ is the volume of unit ball in
$\mathbb{R}^{2n}$.

Note that $A$, $B$, and $C$ defined in (\ref{ABC one form}) can be expressed in terms
of $p$ as
\begin{equation}
A=\frac{p^{\prime}}{2xp^{3}}, \ \ \ B=\frac{x^2}{v^2}-\frac{1}{vp},\
\ \ C=\frac{2}{v}-\frac{2x^2}{v^2}.
\end{equation}

\subsection{Examples in terms of the \texorpdfstring{$p$}{TEXT} function}

\begin{example}[\textbf{$0<p(+\infty)<1$}]
\label{Euclidean}

Any smooth function $p$ on $[0,+\infty)$ with
\begin{align}
p(0)=1,\ \ p^{\prime}(0)=0, \ \ p^{\prime}<0, \ \text{and}\
0<p(+\infty)<1  \label{peuc}
\end{align}
defines a complete K\"{a}hler metric on
$\mathbb{C}^n$ with Euclidean volume growth.  In fact, Lemma \ref{eucgrowth} shows that (\ref{peuc}) characterizes all metrics in $\mathcal{N}_n$ with Euclidean volume growth. Any metric in this class satisfies quadratic average curvature decay, i.e. there exist two constants $C_1$ and $C_2$ such that
\begin{equation}
-\frac{C_1}{{(1+s)}^2} \leq \frac{1}{\operatorname{Vol}(B(O,s))}
\int_{B(O,s)} R(s) w^n \leq -\frac{C_2}{{(1+s)}^2}, \ \ \ \forall s>0. \label{avesense}
\end{equation}
This follows from the proof of Proposition \ref{curvdecay} where we derive some general results on curvature decay.
\end{example}

Example \ref{Euclidean} consists of metrics with its curvature component $A$ going to $-\infty$ along a sequence of points tending to infinity, see Lemma \ref{Aunbound} in below. We also have rough estimates $\limsup_{s
\rightarrow +\infty} s^2 A=0$, $\lim_{s \rightarrow +\infty} s^2
B=0$, and $C \sim -\frac{C_1}{s^2}$ for $s$ large enough. As a simple
example, we may pick $p(x)=\frac{x+1}{2x+1}$ outside a small neighborhood of
$x=0$, which produces a K\"ahler metric with negative sectional curvature
outside a compact set. Moreover, one has $A \sim -\frac{C_1}{s^3}$, $B \sim
-\frac{C_1}{s^3}$, and $C \sim -\frac{C_1}{s^2}$.

Function theory on Example \ref{Euclidean} is completely determined by the asymptotics of $p(x)$.
Let $p(\infty) \doteq \lim_{x \rightarrow \infty} p(x)$. Then the coordinate function $z_i$ has growth like $s^{p(\infty)}$ for large $s$, and the anticanonical section has the polynomial growth
order of $n(1-p(\infty))$. Precisely, we have
\begin{equation}
\lim_{s \rightarrow \infty} \frac{\log |z_i|}{ \log s}=p(\infty), \ \
\lim_{s \rightarrow \infty} \frac{\log \sqrt{\operatorname{det}(g)}}
{ \log s}=n(1-p(\infty)).
\end{equation}
Furthermore, the asymptotic volume growth and the average curvature decay are related as
\begin{equation}
\lim_{s \rightarrow \infty}\frac{\operatorname{Vol}(B(s))}{c_n
s^{2n}}=(\frac{1}{p(\infty)})^{n},
\end{equation}
\begin{equation}
\lim_{s \rightarrow \infty} \frac{s^2}{\operatorname{Vol}(B(O, s))}
\int_{B(s)} (-R) \, {\omega}^n=n^2(1-p(\infty)).
\end{equation}

\begin{example}
[expanding K\"{a}hler-Ricci solitons by Cao
{\cite{C97}}] They belong to Example \ref{Euclidean} and are of particular interest in the study of Type-III solutions to K\"ahler-Ricci flow. Cao \cite{C97} constructed a family of $U(n)$-invariant complete expanding K\"{a}hler-Ricci solitons on $\mathbb{C}^n$. It
is a continuous family of positively curved and negatively curved
metrics connected by the Euclidean one. We only consider those negatively curved ones. Following the calculation in Chen-Zhu \cite{CZ05}, we may check that those metrics are
of negative complex curvature operator. Moreover, it satisfies $A \sim -\frac{C_1}{s^6}$, $B \sim -\frac{C_1}{s^4}$, and $C \sim -\frac{C_1}{s^2}$.
\end{example}

\begin{example}
[\textbf{$p(\infty)=0$}]   \label{pinf=0}

Let $p(x)=\frac{1}{(\log x)^{\alpha}}$ where $0<\alpha \leq 1$ for
large $x$. One can smoothly extend $p(x)$ to $[0,+\infty)$ and it
defines a metric in $\mathcal{N}_n$.

\end{example}

\textbf{Case 1:} $\alpha=1$.

In this case, it follows that $\operatorname{Vol}(B(s))$ grows like
$(s^{2}\log{s})^{n}$ and $|z|$ grows like $\log s$. Moreover, the
curvatures decay as $A \sim -\frac{1}{2s^{2} \log s}, B \sim
-\frac{1}{2s^{2} \log s}$, and  $C \sim -\frac{1}{s^2}$. Also
it is direct to check that the average curvature decays
quadratically.
\begin{equation}
\frac{1}{\operatorname{Vol}(B(O,s))} \int_{B(O,s)} R(s) w^n \sim
-\frac{C_1}{{(1+s)}^2}.
\end{equation}

Note that there are two known examples in this case.

When $n=1$, Milnor \cite{Milnor} gave an example on the complex plane $\mathbb{C}$ which
is $g=ds^2+(s\log s)^2 d\theta^2$ for $s>2$ under geodesic normal
coordinates.

For any $n \geq 1$, Seshadri \cite{Sesha} gave an example of
$U(n)$-invariant K\"{a}hler metric with negative sectional curvature
on $\mathbb{C}^n$. Let $f(r)=e^{r}$. Then
$\xi(r)=-r-1+\frac{1}{r+1}$, $x=\sqrt{r(r+1)} e^{\frac{r}{2}}$,
$p(r)=\frac{r+1}{r^2+3r+1}$, $v=re^{r}$, and $s$ grows like
$e^{\frac{r}{2}}$. So we have $p(x) \sim \frac{1}{2\log x}$. It remains to check that $\omega$ is indeed of negative sectional
curvature. As the result in \cite{Sesha} is based on an incomplete proof of
Klembeck \cite{K}, we argue by Theorem \ref{mainApp} instead. For any $r \geq 0$,
\begin{equation}
A=-\frac{1}{(r+1)e^{r}}-\frac{1}{(r+1)^3 e^{r}}, \,
B=-\frac{1}{e^r(r+1)}, \, C=-\frac{2}{e^r}
\end{equation}
and
\begin{equation}
D_{n}=\frac{n}{2(n-1)} AC-B^2=\frac{1}{e^{2r}(r+1)^2}
[\frac{n}{n-1}(r+1+\frac{1}{r+1})-1]>0.
\end{equation}
Therefore Seshadri's example
has negative complex curvature operator.

\textbf{Case 2:} $0<\alpha<1$.

In this case, $\operatorname{Vol}(B(s))$ grows like
$(s^{2}(\log{s})^{\alpha})^{n}$, and $\log |z|=\frac{(\log
x)^{1-\alpha}}{1-\alpha}$. This implies that for any $\beta>1$ and
$0<\alpha<1$ we have $(\log s)^{\beta}<|z|<s^{\alpha}$ for large
$s$. Its curvature components satisfy $A \sim -\frac{\alpha}{2s^{2} \log s}, B
\sim -\frac{\alpha}{2s^{2} \log s}$, and $C \sim
-\frac{1}{s^2}$.

\subsection{Estimates on volume growth and curvature decay}

For the sake of convenience, we introduce another characterization
of $\mathcal{N}_n$. Let us consider the distance function $
s=\int_0^r \frac{\sqrt{h(t)}}{2\sqrt{t}}dt = \int_{\ 0}^{\sqrt{r}}
\sqrt{h(\tau^2)}d\tau $. Clearly, $s$ is a $C^{\infty }$ change of
variable of $\sqrt{r}$ on $[0,\infty )$. So $r=r(s)$ is a smooth
function and we may define $\psi (s)$ on $[0,\infty )$ by letting
$\psi (s)=-\xi(r(s))$.
\begin{lemma}[\textbf{the $\psi$ characterization}]\label{psilemma}
Any $\psi \in C^{\infty}[0,\infty )$ which satisfies
\begin{align}
\ \ \ \psi (0)=\psi'(0)=0,  \ \ \psi ''(0)>0, \ \ \psi'(s)>0 \ \
\forall s>0, \ \  \int_1^{\infty } \! \frac{1}{s+\phi (s)}ds =
\infty,   \label{psifunction}
\end{align}
where $\phi (s)=\int_0^s \psi (\tau )d\tau $ generates a metric in $\mathcal{N}_n$.
\end{lemma}

\begin{proof}[Proof of Lemma \ref{psilemma}]
Note that the first part is clear, while the last integral is equivalent to
$\int_1^{\infty } \frac{p(x)}{x}dx =\infty$. Conversely, having a
function $\psi (s) \in \Psi$, one can construct a unique function
$\psi (r)$. Indeed we have $\psi
(s) =-\xi (r)$, where $\frac{ds}{dr}=\frac{\sqrt{h}}{2\sqrt{r}}$,
and $h$ is determined by $\xi$ as before.
\end{proof}

So from now on we will use $\psi (s)$ as our generating function for $\mathcal{N}_n$. It is easy to get examples where $A$ is not bounded from above.

\begin{lemma}[perturbations with unbounded curvature]\label{Aunbound}
Given any metric in $\mathcal{N}_n$ with a generating function $\psi$ as in Lemma \ref{psilemma}. We may find a new generating function $\psi_1$ so that the radial curvature $A$ of the corresponding metric is unbounded.
\end{lemma}

\begin{proof}[Proof of Lemma \ref{Aunbound}]
Start from any metric associated with a corresponding $\psi$. Note that $\psi$ in
an increasing function. Take any increasing sequence $\{s_k\}$ going
to infinity, with $s_1>0$. Now let us define a smooth and nonnegative
function $q$ on $[0,\infty )$ such that
$$q(s_k)> 4k ( s_k +  \phi (s_k)),$$
for each $k$,  while $\psi_q :=\int_0^sq \leq \psi$. This is
possible since we may take $q$ to be zero except at a smaller and
smaller neighborhood of each $s_k$. We have $\phi_q :=\int_0^s
\psi_q \leq \phi$. So if we consider the function $\psi_1=\psi +
\psi_q$, we have its $\phi_1=\phi +\phi_q \leq 2\phi$, and thus
$\frac{1}{s+\phi_1}$ has infinite integral over $[1,\infty )$ and
gives a metric in ${\mathcal N}_n$. For this metric, its
corresponding curvature term
\[ -A_1 =\frac{\psi_1'} {2s+2\phi_1}
\geq \frac{\psi '+q}{4s+4\phi } = -\frac{1}{2}A+ \frac{q}{4(s+\phi
)}.\] Therefore  $A_1(s_k) > k$ for each $k$ and is unbounded.
\end{proof}

\begin{remark}
The proof of Lemma \ref{Aunbound} can be used to show that for any given metric $g$ in $\mathcal{N}_n$, there are some metrics with equivalent volume growth and equivalent average scalar curvature decay functions, but the scalar curvature of the nearby metrics goes to infinity at any prescribed rate along a sequence going to infinity. Here two positive smooth functions $F_1$ and $F_2$ on $[a,\infty )$ are said to be \textbf{equivalent}, if there are constants $C_1, C_2>0$ such that $C_1 F_1(s) \leq F_2(s) \leq C_2 F_1(s)$ for any $s\geq a$ large.
\end{remark}

Recall that the geodesic ball centered at the origin $B(O, s)$ satisfies
$\operatorname{Vol}(B(O, s))=c_n v^n$ and $v=rf$ by (\ref{volformula}). Now we have the following basic relation
\begin{align}
v^{\prime}(s)=2\sqrt{hr}, \ \ v^{\prime\prime}(s) = 2 + 2 \psi (s).  \label{voleqn}
\end{align}
This is the main advantage of introducing the $\psi$ characterization for $\mathcal{N}_n$.
Since $\psi \geq 0$, we get $v''\geq 2$ and $v \geq s^2$. So the volume is at least of Euclidean growth. In fact we may prove:

\begin{lemma} [a monotonicity result on volume growth]\label{monotonicity}
For any given metric in $\mathcal{N}_n$, we have
\begin{align}
\frac{v(s)}{s^{2}}, \ \ \text{hence}\ \frac{\operatorname{Vol}(B(O, s))}{\operatorname{Vol}_{\mathbb{R}^{2n}}(B(s))}
\end{align} is increasing with respect to $s>0$.
\end{lemma}

\begin{proof} [Proof of Proposition \ref{monotonicity}]
By (\ref{voleqn}) we solve
\[
\frac{d}{ds} (\frac{v}{s^2})=\frac{2s (s\sqrt{rh}-v)}{s^4}.
\] It suffices to show $s\sqrt{rh}-v=sx-v$ is increasing. To that end we derive with respect to $x$ instead, by (\ref{vbyp})
\[
\frac{d}{dx} (xs-v)=\int_0^x p(\tau)d\tau-xp(x)>0.
\]
In the above we use $p(x)$ is decreasing by Theorem \ref{pfuncthm}.
\end{proof}

We may use (\ref{voleqn}) to characterize metrics with Euclidean volume growth in $\mathcal{N}_n$.
\begin{lemma}[a characterization of Euclidean volume growth]\label{eucgrowth}
Any K\"ahler metric in $\mathcal{N}_n$ is of Euclidean volume growth if and only if any of the following holds
\begin{enumerate}
\item $\liminf_{x \rightarrow \infty} p(x)>0$;
\item $\limsup_{r \rightarrow \infty} -\xi(r)<+\infty$;
\item $\limsup_{s \rightarrow \infty} \psi(s)<+\infty$.
\end{enumerate}
\end{lemma}
\begin{proof}[Proof of Lemma \ref{eucgrowth}]
After integrating the basic equation (\ref{voleqn}), we have
\begin{align}
v= s^2 + 2\int_0^s \phi (\tau )d \tau,\ \ \text{where}\ \phi(s)=\int_0^{s} \psi(\tau)d\tau.  \label{volumeint}
\end{align}
Therefore the Euclidean volume growth condition is equivalent to
\begin{align}
\int_0^s \phi (\tau )d \tau \leq C_1 s^2   \label{euccondi1}
\end{align} for some constant $C_1>0$.

It follows that $\phi(s) \leq C_2 s$ for some constant $C_2>0$. To see it, we first observe $\frac{\phi(s)}{s}$ is increasing with respect to $s>0$.
\[
\frac{d}{ds}(\frac{\phi(s)}{s})=\frac{\psi(s)s-\int_0^s \psi(t)dt}{s} \geq 0,
\] where we use $\psi(0)=0$ and $\psi'(s)>0$ for any $s>0$. If $\limsup_{s \rightarrow \infty} \frac{\phi(s)}{s}=+\infty$, then for any $M>0$, we may find $s_0>0$ so that $\phi(s)> M s$ for any $s>s_0$. This implies
\[
\limsup_{s \rightarrow \infty}  \frac{\int_0^s \phi(\tau)d\tau}{s^2} \geq \limsup_{s \rightarrow \infty}  \frac{\frac{M}{2}(s^2-s_0^2)}{s^2} \geq \frac{M}{2}.
\] As $M$ is arbitrarily large, we get a contradiction with (\ref{euccondi1}).

Now we conclude $\limsup_{s \rightarrow \infty} \psi(s)<C_3$ for some $C_3>0$, as $\phi(s)=\int_0^s \psi(\tau)d\tau \leq C_2 s$.

\end{proof}

Recall in Example \ref{Euclidean} and Example \ref{pinf=0}, the volume growth satisfies
$v \sim s^2 (\ln s)^{\alpha }$, where $0 \leq \alpha \leq 1$. Making use of Theorem \ref{pfuncthm}, it is easy to construct examples where $v \sim s^2 \ln s \ln
\ln s$, or more generally, $v\sim s F_n(s)$, where $ F_n(s)=l_0(s)
l_1(s) l_2(s) \cdots l_n(s)$, with
\[
l_0(s) = s, \ \ l_n(s) = \ln(l_{n-1}(s)).
\]
On the other hand, the following result shows the volume growth in the above cases is in some sense optimal.

\begin{proposition} [a lower estimate on the volume growth]
\label{volest} For any given
$\alpha>1$ and any metric in $\mathcal{N}_n$, we have
\begin{align}
\liminf_{s \rightarrow +\infty} \frac{v(s)}{s^{2}(\ln s)^{\alpha}}=0.
\end{align}
For instance, given any $\epsilon >0$, there is no metric in $\mathcal{N}_n$ with the growth of $\operatorname{Vol}(B(O, s))$ equivalent to $s^{2n}(\ln s)^{n+\epsilon}$.
\end{proposition}

\begin{proof} [Proof of Proposition \ref{volest}]

In fact, for any given positive increasing function $\phi_0(s)$ such that
$\int_1^{\infty}\frac{1}{\phi_0(s)} ds < \infty$, we cannot have
$v\geq \epsilon s \phi_0(s)$ for all $s\geq 1$.

Indeed, it follows from (\ref{volumeint}) that we have
\[
v= s^2 + 2\int_0^s \phi (\tau )d \tau \leq s^2 + 2s\phi (s),
\]
as $\phi$ is increasing. So we get $ \phi_0 \leq 2s+ 2\phi$, and thus $
\frac{1}{s+\phi }\leq \frac{2}{\phi_0}$  having finite
integration over $[1,\infty)$, a contradiction.

\end{proof}

For any $U(n)$-invariant function $H$ on ${\mathbb C}^n$, let us
denote by $\overline{H}= \frac{1}{c_nv^n} \int_{B(s)}H\omega^n $ the
average function of $H$ over the geodesic balls. Then we have
\begin{align}
\overline{H}(s) = \frac{1}{v^n} \int_0^s H d(v^n).  \label{averagedef}
\end{align}

\begin{proposition} [estimates on the average curvature decay]
\label{curvdecay} For any metric in $\mathcal{N}_n$, we have:
\begin{align}
\limsup_{s \rightarrow +\infty} \frac{s^2 }{(\ln s)}
\overline{R}=0, \,\,\ \ \ \  \limsup_{s \rightarrow +\infty} s^2 \overline{R}<0.  \label{avgdecay}
\end{align} Here $\overline{R}$ denote the average (minus) scalar curvature function defined in (\ref{averagedef}).
\end{proposition}

\begin{proof} [Proof of Proposition \ref{curvdecay}]

Recall that the geodesic ball $B(O, s)$ satisfies
$\operatorname{Vol}(B(O, s))=c_n v^n$ by (\ref{volformula}). We introduce
\begin{align}
-A=\frac{\psi '}{v'}, \ \  \  C_0 \doteq -(B+\frac{1}{2}C)=\frac{\psi }{v}.
\end{align}

So we have
\[
\overline{-A}= nC_0 - \frac{n}{n-1} \overline{C_0}, \ \ \
\overline{C_0} \leq \frac{n}{n-1} C_0.
\]
The inequality is because
$\psi $ is an increasing function. We have that (the minus of) the
scalar curvature equal to
\begin{align*}
-R & =-A - 2(n-1)B - \frac{1}{2} n(n-1)C \\
& =  -A + 2(n-1)C_0 - \frac{1}{2}(n-1)(n-2)C\\
& =  -A + n(n-1) C_0 + (n-1)(n-2)B.
\end{align*}
So we have
\begin{align}
-A + 2(n-1)C_0 \leq (-R) \leq  -A + n(n-1) C_0,
\end{align}
and when $n=2$ the above is actually an equality. So when $n\geq 2$,
we have
\begin{align}
nC_0 \leq (\overline{-R}) \leq n(1+ \frac{n^2(n-2)}{(n-1)^2})C_0.
\end{align}
That is, $-R$ is equivalent to $-A+C_0$ and the average scalar
curvature is equivalent to $C_0$. So we will focus on these
functions now.

It is easy to see that $\liminf_{s \rightarrow +\infty} s^2
(\overline{-R})>0$, since from (\ref{voleqn})
\[
\frac{s^2 \overline{-R}}{n} \geq \frac{s^2 \psi}{v} \geq \frac{s^2
\psi}{s^2+s^2 \psi}=\frac{\psi}{1+\psi}.
\]

Assume the first inequality in (\ref{avgdecay}) does not hold, namely, there exists some $\epsilon>0$ so that $\psi (s) \geq \epsilon \frac{\ln s}{s^2}
v(s)$ for any $s\geq 1$. For any $0<a<1$, we have from (\ref{voleqn})
\[
v(s) \geq \int_0^s \phi  \geq \int_{as}^s \phi \geq (1-a)s\phi (as).
\]
Similarly,
\[
\phi (s) = \int_0^s \psi \geq (1-a)s \psi (as).
\]
Therefore, we have
\begin{align*}
\phi (s) & \geq (1-a)s \psi (as) \ \geq \ (1-a) \epsilon
\frac{\ln (as)}{a^2s} v(as) \\
& \geq (1-a)\epsilon \frac{\ln (as)}{a^2s} (1-a)as \phi (a^2s)\\
& = \frac{\epsilon (1-a)^2}{a} \ln (as) \phi (a^2s).
\end{align*}
By taking $a=\frac{1}{\sqrt{e}}$, we get that
\[
\phi (s) \geq b \ \ln s \ \phi (\frac{s}{e})
\]
for any $s\geq e$, where $b>0$ is a constant. For any $s\geq e$,
there exists a unique positive integer $n\geq 2$ such that $e^n\leq
s<e^{n+1}$, or equivalently, $n\leq \ln s<n+1$. By the above
inequality, we get
\begin{align*}
\phi (s) & \geq  b \ \ln s \ \phi(\frac{s}{e}) \\
& \geq  b\ \ln s \ b \ \ln \frac{s}{e} \ \cdots \ b \ \ln \frac{s}{e^{n-1}} \ \phi (\frac{s}{e^n}) \\
& \geq  b^n n! c
\end{align*}
where $c>0$ is a lower bound of $\phi$ over $[1,e]$. By the
Stirling's formula, we have $n! \geq (\frac{n}{e})^n$. Hence
\begin{equation*}
    \phi (s) \geq c \ (\frac{b}{e}n)^n \geq c \ (\frac{b}{2e}\ln s)^{\frac{1}{2}
 \ln s} = c\ s^{\frac{1}{2} \ln (\frac{b}{2e}\ln s) } .
\end{equation*}
So for sufficiently large constant $K>0$, we would have $\phi (s)
\geq c  s^2$ for any $s\geq K$. Hence $\frac{1}{\phi }$ will have
finite integral over $[K,\infty )$, a contradiction.

\end{proof}

In contrast, we cannot expect upper estimates on volume growth and curvature for a general metric in $\mathcal{N}_n$. In the following, we show that given any increasing convex function $F(s)$ on $[0,\infty )$, there are metrics in $\mathcal{N}_n$ whose volume growth or average (minus) scalar curvature decay are greater than $F$ along some sequence of points $\{s_k\} \rightarrow \infty$. As before, we use the $\psi$ characterization of $\mathcal{N}_n$ obtained in Lemma \ref{psilemma}. Given any $\psi$ satisfying (\ref{psifunction}), there is a corresponding metric $g$ in $\mathcal{N}_n$ up to a positive constant multiple. Recall $\phi (s)=\int_0^s \psi (\tau )d\tau$. Then (\ref{psifunction}) can be restated for $\phi$. Namely, a smooth functions $\phi(s)$ on $[0, \infty)$ such that
\begin{align}
& \phi (0)=\phi '(0)=\phi^{\prime\prime}(0)=0, \phi^{\prime\prime\prime}(0)>0;  \nonumber\\
& \phi '(s)>0, \phi^{\prime\prime}(s)>0,\ \ \forall \ s>0,\ and\ \int_1^{\infty } \frac{1}{s+\phi(s)} ds = \infty.  \label{phicondi}
\end{align}
will generate a metric $g_{\phi}$ in $\mathcal{N}_n$.



The following result shows that one cannot expect an upper estimate on volume growth.

\begin{proposition}[no pointwise upper volume estimate]
\label{uppervolume}
Let $F(s)$ be any positive function on $[0, \infty )$. Then there exists a metric $g_{\phi }$ in $\mathcal{N}_n$ and a sequence $\{ s_k\} \rightarrow  \infty $ such that $v(s_k) > F(s_k)$ for each $k$.
\end{proposition}

\begin{proof} [Proof of Proposition \ref{uppervolume}]

It suffices to construct a function $\phi$ which satisfies (\ref{phicondi}) and
\[
\int_0^{s_k}\phi(\tau)d\tau > F(s_k)
\] for a sequence $\{s_k\} \rightarrow \infty $.

We begin with the following observation. Given any $a, b$, and $\lambda >0$, let $f(s)=\epsilon (s-a)^2+\lambda (s-a)+b$ on $(a, +\infty)$. Then for $\epsilon>0$ sufficiently small, there exists some $t>a$ such that $\int_a^t \frac{1}{f}=1$.

Now let us construct a piecewise quadratic function $f$ on $[1,\infty)$. Let $(a,b)=(s_1,f_1)=(s_1,2F(s_1))$ and $\lambda =\lambda_1>3$. Choose $\epsilon_1>0$ sufficiently small and then there exists $t_1>s_1$ such that $\int_{s_1}^{t_1}\frac{1}{f}=1$ for $f(s)=\epsilon_1 (s-s_1)^2+\lambda_1 (s-s_1)+f_1$. Let $s_2=2t_1$, and choose $\mu_1> \lambda_1+2\epsilon_1(t_1-s_1)$ such that $\mu_1> \frac{2}{s_2^2}F(s_2)$. Let $f(s)= (s-t_1)^2+\mu_1(s-t_1)+f(t_1)$ on $[t_1, s_2]$. Then we have $\int_{t_1}^{s_2}f > F(s_2)$, and the two one-sided derivatives of $f$ satisfies $f'_-(t_1) < f'_+(t_1)$. Take  $\lambda_2$ so that $\lambda_2 > s_2+\mu_1$. Repeat the process for $(a,b)=(s_2, f(s_2))$ and $\lambda_2$. We get parabolas on $[s_2,t_2]$ and $[t_2, s_3]$, where $s_3=2t_2$. Continuing with the process, we obtain a strictly increasing, piecewise quadratic function $f$ on $[1,\infty )$ which is strictly convex in each interval $[s_k, s_{k+1}]$ and with left-hand derivative strictly less than right-hand derivative at each break point $t_k \in (s_k, s_{k+1})$. Furthermore,
\begin{align*}
\int_1^{\infty }\frac{1}{f(s)}ds >\sum_{i=1}^{\infty} \int_{s_i}^{t_i} \frac{1}{f}=\infty,\ \ \ \int_1^{s_k} f(\tau)d\tau>\int_{t_{k-1}}^{s_{k}} f> F(s_k).
\end{align*}

Clearly one can smooth out the corners $\{t_k\}_{k=1}^{\infty}$ of $f$ to get the desirede function $\phi$ with $\int_1^{s_k} f(\tau)d\tau>F(s_k)$ so that all conditions in (\ref{phicondi}) are fulfilled.
\end{proof}

By Proposition \ref{curvdecay}. the average scalar curvature function $\overline{R}(s)$ is bounded from above by $-\frac{1}{s^2}$. Moreover, there exists no constant $\epsilon >0$ such that $\overline{R}(s) \leq -\epsilon \frac{\ln s}{s^2}$ for all $s\geq 1$. Following the proof of Proposition \ref{uppervolume}, we show that one cannot expect a pointwise lower bound on the average scalar curvature.

\begin{proposition}[no lower estimates on average curvature]\label{uppercurvature}
Let $F(s)$ be any positive function on $[0, \infty )$. Then there exists a metric $g$ in $\mathcal{N}_n$ and a sequence $\{ s_k\} \rightarrow  \infty $ such that $\frac{\psi(s_k)}{v(s_k)}  > F(s_k)$ for each $k$. Namely $\overline{R}(s_k)$ approaches to $-\infty$ as $\{ s_k\} \rightarrow  \infty $ even though $\limsup_{s \rightarrow \infty} \frac{s^2}{\ln s}\overline{R}(s)=0$.
\end{proposition}

Next we discuss the correlation between volume growth and averaged scalar curvature decay. For $U(n)$-invariant K\"ahler metrics of $BI>0$ on $\mathbb{C}^n$, there is a precise connection (see \cite[Theorem 7 on p.538]{WZ}). Unfortunately the situation becomes more subtle for $BI<0$. We recall Example \ref{Euclidean}. If a metric in $\mathcal{N}_n$ has Euclidean volume growth, then its average scalar curvature is of quadratic decay. Namely, if $v \sim s^2$, then $\frac{\psi }{v} \sim \frac{1}{s^2}$. We present another proof of this fact. Notice that $\phi$ is convex. So we have
\begin{align}
\phi (s) =\phi (\frac{0+2s}{2}) \leq \frac{1}{2s}\int_0^{2s} \phi(\tau) d\tau \leq \frac{1}{2}(\phi (0)+\phi (2s)).  \label{convexcontrol}
\end{align}
As $v\sim s^2$, it follows from (\ref{voleqn}) that there exists a constant $C$ such that $\int_0^s\phi \leq Cs^2$, so the above (\ref{convexcontrol}) gives $\phi (s) \leq 2Cs$. That is, the graph of the function $\phi$ is underneath the line $y=2Cs$ on the $sy$-plane, but it must lie above all its tangent lines since it is convex. So the slope of the tangent line has to be less than or equal to $2C$, namely, $\psi \leq 2C$. Thus $C_0=\frac{\psi }{v}\leq \frac{2C}{v} \leq \frac{2C}{s^2}$, and hence $C_0\sim \frac{1}{s^2}$. Unfortunately the converse is not valid. Namely there are metrics in $\mathcal{N}_n$ with $v \sim s^2\ln s$ and $C_0\sim \frac{1}{s^2}$, as we have seen in Example \ref{pinf=0}.

Finally, we note that it is possible to obtain some correlation between volume growth and curvature decay under additional assumptions.

\begin{proposition}\label{upperassumingbound}
Suppose that a metric in $\mathcal{N}_n$ satisfies either of
\begin{enumerate}
\item The radical curvature component $A$ is bounded;
\item The average scalar curvature $\overline{R}$ defined in (\ref{averagedef}) is bounded.
\end{enumerate}
Then the volume growth is at most exponential, i.e., there exists a constant $C>0$ such that $v(s) \leq Ce^{C\,s}$ for any $s \geq 1$.
\end{proposition}

\begin{proof}[Proof of Proposition \ref{upperassumingbound}]
Since $v'=2s+2\phi$, we have $\psi = \frac{1}{2}(v''-1)\leq Cv$. Hence $v^{\prime\prime} \leq 2+2Cv$. Multiply on both sides by the positive function $2v'$ simultaneously and integrate. Then $ v'^2\leq 2v+2Cv^2\leq 2(C+1)v^2$ for any $s\geq 1$. Write $C_1^2=2(C+1)$, and we get $v'\leq C_1v$ for any $s\geq 1$, which leads to $v\leq C_2e^{C_1s}$.

If we assume that $-A=\frac{\psi '}{v'}\leq C$ instead, then $v'''\leq 2Cv'$. After multiplying by $2v''>0$ and integrating, we get $v \leq \frac{1}{C}e^{\sqrt{2C}s}-\frac{1}{C}-\frac{2}{\sqrt{2C}}s$.
\end{proof}

\section{Examples without holomorphic functions of slow growth}  \label{noholoslow}

In this section, we study examples of complete metrics on the complex plane $\mathbb{C}$ with negative Gauss curvature ($K \leq 0$ for short). The goal is to illustrate some non-standard behaviors in the growth of holomorphic functions.

We begin with
\[
\lambda(x)=\begin{cases}
\frac{1}{(1+x)^2}, \ \ \text{when}\ \ x>0;   \\
e^{x^2-2x},\ \ \text{when}\ \ x \leq 0.
\end{cases}
\]
We may check that $\lambda$ is a $C^2$ function with $\log \lambda$ convex on $\mathbb{R}$.

Similarly, we can construct $\lambda$ which is smooth and log-convex on $\mathbb{R}$. For example, we choose $\lambda(x)= \frac{1}{(1+x)^2}$ when $x>-\frac{1}{2}$. Then we find
\[
\log \lambda (-\frac{1}{2})=2\log 2, \ (\log \lambda)'(-\frac{1}{2}+)=-4.
\]
Thus $y=-4x-2+2\log 2$ is a tangent line of $\log \lambda$ at $x=-\frac{1}{2}$. When $x<-\frac{1}{2}$ we let $\lambda(x)=e^{A(x+\frac{1}{2})^2-8x-4+2\log 2}$ for $A>0$ large enough. Then $\lambda$ is smooth except one point $x=-\frac{1}{2}$. Note that we have $\lambda'(-\frac{1}{2}-)>\lambda'(-\frac{1}{2}+)$. Therefore we may mollify it on $(-\frac{1}{2}-\varepsilon,-\frac{1}{2}+\varepsilon)$ so that the resulting function is smooth and log-convex on $\mathbb{R}$. For simplicity, we still use $\lambda$ to stand for its mollification.

\begin{proposition}\label{exp_ex1}
Let $g=\lambda(x) (dx^2+dy^2)$ on $\C$. Then
    \begin{enumerate}
    \item  $g$ is complete with $K(g)<0$.
    \item   Any nonconstant holomorphic function grows at least exponentially in the sense that
    \begin{equation}
    \sup_{B_g(O, r)} \log (1+|z|)  \geq r, \ \ \ \forall r>0.  \label{expogrowth1}
    \end{equation}

    \item  For any $r>0$ large, $\operatorname{Vol}(B_g(O,r)) \ge C_1 e^{C_2 r}$ for some constants $C_1, C_2$.
    \end{enumerate}
\end{proposition}

We begin with a simple observation.

\begin{claim} \label{minigeo}
Let $A=(p,0)$ and we consider the distance between the origin $O$ and the curve $L_A=
\{x=p\}$. Then
\[
d_g(O,L_A) \doteq \inf_{Q \in L_A} d_g(O, Q) =d_g(O,A) =L_g(\overrightarrow{O A}).
\]
In particular, the segment $\overrightarrow{O A}$ is a minimal geodesic.
\end{claim}

To see the above claim is true, we consider any curve $\alpha: [0,1] \to \mathbb{C}$ that connects the origin $O$ and any point $Q \in L_A$. Let $\alpha(t)=(x(t), y(t))$. Then
\begin{align*}
L_g(\alpha(t))&= \int_0^1 \sqrt{\lambda(x(t))} \sqrt{(x'(t))^2 +(y'(t))^2}\,dt \\
&\ge \int_0^1 \sqrt{\lambda(x(t))} |x'(t)|\,dt \\
&\ge \int_0^p \sqrt{\lambda(x)} \,dx=L_g(\overrightarrow{O A}).
\end{align*}

\begin{proof}[Proof of Proposition \ref{exp_ex1}]

Step 1: We show $g=\lambda(x)(dx^2+dy^2)$ is complete.

Recall that $\lambda$ is decreasing and $\lambda>1$ when $x<0$. We only need to consider the right half plane. A direct computation:
\[
L_g(\overrightarrow{O A})= \int_0^p \sqrt{\lambda(t)} \,dt= \int_0^p \frac{1}{1+t}\,dt= \log (1+p).
\]
Let $\{z_n=(x_n,y_n)\}_{n=1}^{\infty}$ be any unbounded (with respect to the Euclidean topology) sequence on $\mathbb{C}$. We prove that $\{d_g(O, z_n)\}$ is unbounded. By Claim \ref{minigeo}, if $x_n\to \infty$, we already have $d_g(O,z_n) \to +\infty$. It suffices to consider the case that
$\{x_n\}$ is bounded while $y_n \to \infty$. Suppose $\{d_g(O,z_n)\}$ is bounded. Then for any $\epsilon>0$, there exists a curve $\gamma_n(t)=(x_n(t),y_n(t)): [0,1] \rightarrow \mathbb{C}$ from $O$ to $z_n$ with $L_g(\gamma_n(t)) \leq d_g(O,z_n)+\epsilon$. It follows from Claim \ref{minigeo} that there exists some $M>0$ with $\sup_{t\in[0,1], \forall n} |x_n(t)| \leq M$. Hence
\[
L_g(\gamma_n(t))=\int_{\gamma_n} \sqrt{\lambda(x(t))}|\gamma_n'(t)|\,dt \ge \frac{1}{1+M}L_0(\gamma_n),
\]
where $L_0$ is the Euclidean length. It is impossible since $L_0(\gamma_n)  \geq L_0(\overrightarrow{O z_n})\rightarrow\infty$ as $y_n \to \infty$.

We observe that the curvature of $g$ on the right half plane is constant:
\begin{equation*}
    K=-\frac{\Delta \log \lambda}{2\lambda}=-\frac{(\log \lambda)''}{2\lambda}=-1.
\end{equation*}
For given $\lambda$, such as $e^{A(x+\frac{1}{2})^2-8x-4+2\log 2}$ when $x<-\frac{1}{2}$, $K \to 0$ when $x \to -\infty$.

Step 2: Now that $g$ is complete,  (\ref{expogrowth1}) follows from Claim \ref{minigeo}.

Step 3: We study the volume growth of geodesic balls of $g$.

Consider a triangle $D_a=\{0\le x\le a, |y|\le a-x \}$. Let $\log (1+a)=r$. Then $(a,0) \in \partial B_g(O,r)$ and $D_a \subset B_g(O,(\sqrt 2 +1)r)$.
 \begin{equation}
\begin{aligned}
\operatorname{Vol}D_a&=\int_0^a \int_{-a+x}^{a-x} \frac{1}{(1+x)^2}\,dxdy\\
&=2a(1-\frac{1}{1+a})-2(\log (1+a)-\frac{a}{1+a})\\
&=2(e^r-1)-2r.
\end{aligned}\nonumber
\end{equation}
Therefore we conclude $\operatorname{Vol}B_g(O,(\sqrt 2 +1)r) \ge Ce^r$.
\end{proof}

Based on Proposition \ref{exp_ex1}, we may construct another complete metric $\widehat{g}$ on $\mathbb{C}$ with negative curvature so that $K(\widehat{g})(p) \rightarrow 0$ for any $d(p, O) \rightarrow +\infty$ if $p$ is outside from a sector at the origin with an arbitrarily small openning angle. Take a piecewise smooth function, for example, $\mu(y)=e^{|y|-1}$ when $|y|\ge 1$ and $\mu(y)=1$ for $|y|\le 1$. For an abuse of notation, let $\mu$ denote a mollification of $\mu$ near $|y|=1$ so that $\mu \geq 1$ and $(\ln \mu)^{\prime\prime} \geq 0$ on $\mathbb{R}$. Obviously we may choose $\mu$ as an even function. Now we consider
\[
\widehat g=\mu g=\lambda(x) \mu(y) (dx^2+dy^2).
\]
As $\widehat{g} \geq g$, $\widehat{g}$ is also a complete metric. $\log (\lambda \mu)$ is actually subharmonic (in fact it is even convex). Compute the Gauss curvature on the right half plane outside the strip $\{|y|\le 1+\delta\}$ for some $\delta>0$ small enough:
\begin{align*}
K(\widehat{g})=-\frac{\Delta \log (\lambda\mu)}{2\lambda \mu}=-\frac{1}{e^{|y|-1}}.
\end{align*}
We may further modify $\mu$ so that the height of the horizontal strip $\{x>0,\ |y| \leq 1\}$ is smaller than any given positive number. In the spirit of Proposition \ref{exp_ex1}, we show:

\begin{proposition}\label{exp_ex2}
The complete K\"ahler metric $\widehat{g}=\lambda(x) \mu(y) (dx^2+dy^2)$ on $\mathbb{C}$ satisfies the following properties.
\begin{enumerate}
    \item  $K(\widehat{g}) < 0$ everywhere. Moreover, $K(p) \rightarrow 0$ as $d_{\hat{g}}(p, O) \rightarrow \infty$ as long as $p \in \mathbb{C} \setminus E$ where $E$ consists of two horizontal strips defined by $\{z=x+yi\ |\ x>0, \   |y \pm 1|<\delta\}$ for any given $\delta \in (0, \frac{\pi}{4})$.
    \item  $\widehat{g}$ does not have finite total curvature.
    \item  Any nonconstant holomorphic function is at least of exponential growth.
    \item  The volume of geodesic balls satisfies $\operatorname{Vol}(B_{\widehat{g}}(O,r)) \geq C e^r$ for some constant $C>0$.
\end{enumerate}
\end{proposition}

\begin{proof}[Proof of Proposition \ref{exp_ex2}]
We show that $\hat{g}$ is not of finite total curvature. $$\widehat g=\lambda(x) \mu(y) (dx^2+dy^2)=\frac{e^{y-1}}{(1+x)^2}(dx^2+dy^2),$$
holds on $\{x>0,y>1\}$.
\[
K(\widehat{g})\,dV_{\widehat{g}}=-\frac{\Delta \log (\lambda\mu)}{2\lambda \mu}\lambda\mu\,dxdy=-\frac{1}{(1+x)^2}\,dxdy.
\]
Hence
$$\int_{\{x>0,y>1\}} -K(\widehat{g})\,dV_{\widehat{g}}=\int_1^{+\infty}\int_0^{+\infty}\frac{1}{(1+x)^2}\,dxdy=\int_1^{+\infty}1 \,dy=+\infty.$$

Next we study the volume growth of $\widehat g$. On $\{x\ge 0, y\ge 1\}$, $\lambda(x) \mu(y)=\frac{e^{y-1}}{(1+x)^2}$. Consider the curve such that $\lambda(x) \mu(y)=1$. Define $D_a=\{0\le x\le a,\ y \geq 1,\ \lambda(x) \mu(y)\le 1 \}$. Let $r=\log (1+a)$.
\begin{equation}
\begin{aligned}
\operatorname{Vol}D_a&=\int_0^a \int_1^{1+2\log (1+x)} \frac{e^{y-1}}{(1+x)^2}\,dxdy\\
&=\int_0^a \frac{(1+x)^2-1}{(1+x)^2} \,dx\\
&=a-\frac{a}{1+a}= e^r-2+e^{-r}.
\end{aligned}\nonumber
\end{equation}
On the other hand, for any $0<b\le a$,
\begin{equation*}
    \int_1^{1+2\log (1+x)}\sqrt{\frac{e^{y-1}}{(1+b)^2}}\,dy=\frac{2x}{1+b}.
\end{equation*}
Combined with the triangle inequality, for any point $p=(b, y_0) \in D_{a}$,
\[
d_{\hat{g}}(P, O) \leq \frac{2b}{1+b}+\frac{1}{1+b}+\ln(1+b) \leq \ln(1+a)+2=r+2.
\]
Hence $D_a \subset B_{\widehat{g}}(O,r+2)$.
\end{proof}

\begin{remark}
    We expect an example of nonpositive curved complete K\"ahler metric on $\mathbb{C}$ so that all conclusions in Proposition \ref{exp_ex2} holds and its Gauss curvature satisfies $\lim_{p \rightarrow \infty} K(p)=0$.
\end{remark}

Recall the Hadamard order of $f \in \mathcal{O}(\mathbb{C})$ on $(\mathbb{C}, g)$ is defined as
\begin{align}
    \operatorname{Ord}_{H} (f)=\limsup_{p \rightarrow \infty}  \frac{\log \log (|f(p)|)}{\log d_g(p, O)}.
    \label{Horderdef}
\end{align}
We may construct another example such that the Hadamard order of the coordinate function $z$ is $\infty$. Let $u(x,y)=\dfrac{1}{x^2(\log x)^2}$ when $x\ge e$. A direct computation:
\begin{align*}
    &(\log u)'=-2\frac{1+\log x}{x \log x}, \\
    &(\log u)''=2\frac{1+\log x+(\log x)^2}{x^2 (\log x)^2}>0.
\end{align*}

Now $(\log u)'_+(e)=-\frac{4}{e}$, $(\log u)''_+(e)=\frac{6}{e^2}$. Then we can extend $\log u$ to a $C^{\infty}$ subharmonic function on $\C$. $g=u(x,y)(dx^2+dy^2)$ is a complete metric with $K\le 0$.
The proof of completeness and the following inequality are similar to Proposition \ref{exp_ex1}.
\begin{equation*}
    d_g(O,(x,0))= \int_0^x \sqrt{u(t,0)}\,dt = \int_e^x \dfrac{1}{t\log t}\,dt+C=\log\log x+C.
\end{equation*}
Hence when $r>0$ large enough,
\begin{equation*}
    \sup_{z\in B_g(O,r)} \log\log|z| \ge r-C.
\end{equation*}
It follows that
\begin{equation*}
    \operatorname{Ord}_{H}|z|=\limsup_{r\to \infty} \frac{\log\log\sup_{z\in B_g(O,r)} |z|}{\log r} \ge \frac{r}{\log r} =\infty.
\end{equation*}
Note that the Gauss curvature along the $x$-axis
\begin{equation*}
    K=-\frac{\Delta \log u}{2u}=-(1+\log x+(\log x)^2) \to -\infty.
\end{equation*}

It is plausible that an answer to the following question could illuminate the function theory on $\mathbb{C}$ with nonpositive curvature.

\begin{question}
    Let $g$ be a complete K\"ahler metric on the complex plane with its Gauss curvature $-c \leq K(g) \leq 0$ for some constant $c>0$. Is it true that there exists some nonconstant $f \in \mathcal{O}(\mathbb{C})$ with $\operatorname{Ord}_{H} (f) \leq 1$?
\end{question}

\bibliographystyle{acm}

\bibliography{neg_curved}

\end{document}